\documentclass[12pt]{amsart}
\usepackage{amsmath,amsthm,amssymb,amsrefs,bbm,color,esint,esvect,float,graphicx,mathrsfs}
\usepackage[utf8]{inputenc}
\usepackage{fullpage}
\usepackage{hyperref}
\usepackage{euscript,latexsym}
\usepackage{accents}
\title{Endpoint Estimates for Bergman Commutators and New Characterizations of the Bloch Space and $H^\infty$}
\author{Adam B. Christopherson, Zhenghui Huo, Nathan A. Wagner, and Yunus E. Zeytuncu}
\thanks{\begin{footnotesize} Nathan A. Wagner was supported by National Science Foundation grants DMS 203272 and 2549719.\end{footnotesize}}
\date{}
\newtheorem{theorem}{Theorem}[section]
\newtheorem{lemma}[theorem]{Lemma}
\newtheorem{proposition}[theorem]{Proposition}
\newtheorem{corollary}[theorem]{Corollary}
\newtheorem*{corollary*}{Corollary}

\theoremstyle{definition}
\newtheorem{definition}[theorem]{Definition}

\theoremstyle{remark}

\newtheorem{remark}[theorem]{Remark}

\theoremstyle{plain}
\newtheorem{atheorem}{Theorem}

\makeatletter
\newcommand{\alphprime}[1]{%
  \ifcase#1
  \or A
  \or A'
  \else
    \@Alph{\numexpr#1-1\relax}
  \fi
}
\makeatother

\newcommand{\C}{\mathbb{C}}
\newcommand{\R}{\mathbb{R}}

\newcommand{\D}{\mathbb{D}}
\newcommand{\B}{\mathbb{B}_n}

\newcommand{\LlogL}{L \textnormal{ log}^{+} \,L }

\newcommand{\osc}{\operatorname{osc}}
\newcommand{\Bl}{\mathcal{B}}
\newcommand{\BMO}{\rm{BMO}}
\newcommand{\BO}{\rm{BO}}
\newcommand{\wK}{\widehat{K}}

\begin{document}

\maketitle

\begin{abstract}
 We prove an $\LlogL $-type distributional inequality for the commutator of the Bergman projection with a conjugate Bloch symbol function on the unit ball. Such an inequality can be seen as a Bergman version of a result due to C. P\'{e}rez for real-variable Calder\'{o}n-Zygmund operators and BMO functions. We also prove that this inequality characterizes membership of analytic functions in the Bloch space and is further equivalent to a kind of modified restricted weak-type estimate, where one only tests over characteristic functions of sets comparable to Bergman balls. We also show our estimate is sharp in the sense that there exists a Bloch function $b$ so that the commutator $[\bar{b},P]$ is not weak-type $(1,1)$, and prove $[\bar{b},P]$ with $b$ analytic is weak-type $(1,1)$ if and only if $b \in H^\infty$.
\end{abstract}

\section{Introduction}
Let $\B$ be the unit ball in $\C^n$ and let $L^2(\B)$ denote the Hilbert space of square-integrable functions on $\B$ with respect to Lebesgue volume measure $V$. The \emph{Bergman space} $A^2(\B)$ is the closed subspace of $L^2(\B)$ consisting of holomorphic square-integrable functions and the \emph{Bergman projection} $P: L^2(\B) \rightarrow A^2(\B)$ is the orthogonal projection onto the Bergman space. It is well-known that $P$ has the integral representation for $f \in L^2(\B)$:
$$ Pf(z)= c_n \int_{\B}\frac{f(w)}{(1-\langle z, w \rangle)^{n+1}}\, dV(w), \quad z \in \B.$$ Here, $c_n= \frac{n!}{\pi^n}$ denotes a normalizing constant.
It is immediate that $P$ is a bounded linear operator on $L^2(\B)$, but less trivial that it extends to a bounded operator $L^p(\B)$ in the reflexive range $1<p<\infty$. It is well-known (see, for example, \cite{Bekolle} or \cite{DengLiZhao}) that $P$ fails to extend boundedly to $L^1(\B)$, but instead it satisfies a replacement weak-type $(1,1)$ estimate. The mapping properties of the Bergman projection on various function spaces and domains have been a major area of research in analysis for a long time; see, for example, \cite{Barrett, Bekolle, BekolleBonami, ForelliRudin, McNealStein,NRSW,PS}. It is connected to mathematical areas as diverse as decoupling in harmonic analysis and the regularity of the $\bar \partial $-Neumann operator in several complex variables and partial differential equations. 

Motivated by questions in harmonic analysis and operator theory, researchers have also studied the \emph{commutator of the Bergman projection} with a measurable function $b$, defined $[b,P]f:= bPf-P(bf)$. If we assume $b \in L^2(\B)$, this operator is at least densely defined on $L^2(\B).$ The question then becomes: characterize the symbols for which $[b,P]$ extends to a bounded operator on $L^p(\B)$, $1<p<\infty$. It is straightforward the commutator is bounded on every $L^p(\B)$ for bounded symbols $b \in L^\infty(\B)$, so we seek a condition weaker than bounded. The harmonic analysis model to keep in mind is the Hilbert transform $H$ on the real line. It is a celebrated result of Coifman, Rochberg, and Weiss \cite{CRW} that $[b,H]$ is bounded on $L^2$ if and only if $b$ belongs to the space of \emph{bounded mean oscillation} or $\BMO.$ This space is defined on $\mathbb{R}$ by the oscillation condition

$$ \sup_{I \subset \mathbb{R}}\frac{1}{|I|}\int_{I}|b(x)-\langle b \rangle_I|  \, dx, $$
where $I$ denotes an interval on the real line, $|I|$ represents its Lebesgue measure, and $\langle b \rangle_I:= \frac{1}{|I|} \int_{I} b(x) \, dx $ represents the integral average of $b$ on $I$.

In the case of the Bergman projection, there is a well-known connection of commutators to an important class of operators known as Hankel operators. These operators are perhaps best known in the context of the Hardy space on the unit disk $H^2(\D)$, because of their representation by Hankel matrices, but they have also been extensively studied in the Bergman space case. Given a symbol $b$, the Hankel operator $H_b$ is (formally) defined acting on a function $f \in L^2(\B)$ via $H_bf:= (I-P)bPf=bPf-P(bPf)$. Note that $H_b$ maps onto the orthogonal complement of the Bergman space, $A^2(\B)^\perp$. Then one has the (formal) operator identities

$$ [b,P]= H_b -  H_{\bar{b}}^*, \quad H_b = [b,P] \circ P, \quad H_{\bar{b}}^*=P \circ [b,P].$$

So at least on $L^2(\B)$ (and in fact on $L^p(\B)=1<p<\infty$), the bound for the commutator $[b,P]$ is equivalent to simultaneous bounds for $H_b$ and $H_{\bar{b}}^*$. On $L^2$ this statement has a further simplification because of duality, and we can replace the condition ``the adjoint $H_{\bar{b}}^*$ is bounded" by simply ``$H_{\bar{b}}$ is bounded." This connection between Bergman commutators and Hankel operators mirrors a connection to Hardy space Hankel operators for the Hilbert transform.

In the special case that $b$ is holomorphic and belongs to $L^2(\B)$, one can check that $H_{b}$ uniquely extends to the zero operator on $L^2(\B),$ so the classes of symbols for which $[b,P]$ and $H_{\bar{b}}$ are bounded on $L^2(\B)$, respectively, are one and the same. The characterization of such symbols was obtained in an important paper by Axler \cite{Axler}, who showed that these operators are bounded on $L^2(\B)$ if and only if the symbol $b$ belongs to the classical Bloch space $\mathcal{B}$. As we will see, the Bloch space can equivalently be described as the holomorphic functions on $\B$ which satisfy a bounded mean oscillation condition, demonstrating the strong analogy with the Hilbert transform case. Axler also proved that the operator $H_{\bar{b}}$ is compact if and only if $b$ belongs to the little Bloch space. The condition that $b$ is holomorphic was dropped by Zhu in \cite{ZhuHankel}, where he proved that the Hankel operators $H_b$ and $H_{\bar{b}}$ are simultaneously bounded if and only if $b$ belongs to a space of bounded mean oscillation that depends on the exponent $p$, which Zhu denoted by $\BMO_p.$ Generalizations of these results were obtained for bounded symmetric domains and pseudoconvex domains in several complex variables in \cite{BBCZ, Li92, Li94, LiLuecking, HHLPW}.

Let us turn back to the commutator of the Hilbert transform (or more general Calder\'{o}n-Zygmund operators in real-variable harmonic analysis).
The situation at the endpoint $p=1$ for commutators is markedly different from the range $1<p<\infty.$ It is not surprising that if $b$ is a BMO symbol, then the commutator $[b,H]$ is not bounded on $L^1(\B)$. Instead, one might naively expect that the commutator $[b,H]$ obeys a weak-type $(1,1)$ estimate, in line with the operator $H$ itself. However, P\'{e}rez \cite{Perez} gave an example showing that the weak-type estimate fails. The substitute estimate here is an $\LlogL$ distributional inequality:

\begin{theorem}[P\'{e}rez]
Let $T$ be a Calder\'{o}n-Zygmund operator on $\R^n$ and $b \in \BMO.$ Then
there exists $C>0$ so that the following holds for all bounded, compactly supported functions $f$ on $\R^n$ and $\lambda>0$:

$$ \left| \left \{ x \in \R^n: |[b,T]f(x)|>\lambda \right \}  \right | \leq C  \int_{\R^n} \frac{|f(x)|}{\lambda} \left( 1+ \log^{+}\left( \frac{|f(x)|}{\lambda}\right) \right) \, dx.$$
    
\end{theorem}
In the same paper, P\'{e}rez also obtained refined estimates using so-called exponential BMO spaces. These results represent one part of a much larger literature concerning characterizations of $\BMO$ and the behavior of commutators of singular integrals on various function spaces. See, for example, \cite{GLW}.

We turn back to the commutator of the Bergman projection, where comparatively little is known when $p=1.$ We first consider some simple results that we can obtain immediately without effort in this setting. If we assume the unnecessarily strong condition that $b$ is a bounded measurable function, the following simple result for the commutator is an immediate consequence of the weak-type $(1,1)$ estimate for the Bergman projection $P.$

\begin{proposition} \label{LInfinityTrivial}
If $b \in L^{\infty}(\B)$, then the commutator $[b,P]$ is weak-type $(1,1)$.



    
    
\end{proposition}





If we assume the even stronger condition that $b$ is bounded and compactly supported, we get that the commutator is bounded on $L^1(\B).$

\begin{proposition}
Suppose $b$ is bounded and compactly supported. Then the commutator $[b, P]$ is bounded on $L^1(\B).$

\begin{proof}
The integral kernel of the operator $[b,P]$ is equal to $K(z,w)= c_n\frac{(b(z)-b(w))}{ (1-\langle z, w\rangle)^{n+1}}.$ This kernel clearly satisfies the estimate $\sup_{z,w \in \B} |K(z,w)|<\infty$
because $b$ is bounded and compactly supported and the Bergman kernel is bounded away from the boundary diagonal. Then the $L^1$ bound is a trivial consequence of Fubini's theorem.
\end{proof}

\end{proposition}






The goal of this paper is to obtain endpoint estimates for the commutator of the Bergman projection when the strong condition that $b$ is bounded is replaced by some weaker condition, such as bounded mean oscillation. Because of technical obstructions that arise when considering general BMO symbols, this problem is most tractable in the case of anti-holomorphic symbols, which first appeared in Axler's characterization of the Bloch space. Thus, we will restrict our attention to symbols whose conjugate belongs to the Bloch space. Our results can therefore also be viewed as an endpoint version of the $L^2$ bound obtained by Axler. 

First, one might reasonably wonder if it is possible to obtain a weak-type $(1,1)$ estimate for the commutator $[\bar{b}, P]$ when $b$ is a Bloch function, despite the result by P\'{e}rez for the Hilbert transform. One reason is that the anti-holomorphic restriction on $b$ is rather strong, another is that sometimes the Bergman projection is better behaved than singular integral operators.  Our first main result shows that this is not the case.

\begin{atheorem} \label{WeakTypeCounterexample}
There exists a function $b \in \mathcal{B}$ such that the commutator $[\bar{b},P]$ is not weak-type $(1,1).$ In fact, $[\bar{b},P]$ is not even restricted weak-type $(1,1).$
\end{atheorem}
Theorem \ref{WeakTypeCounterexample} can be extended to any bounded pseudoconvex domain $\Omega$ with strongly convex smooth boundary points.    This setting covers all bounded smooth pseudoconvex domains since strongly convex points form a nonempty open subset on the boundary. 
\begin{atheorem} \label{StronglyConvex} Let $\Omega$ be a bounded pseudoconvex domain in $\mathbb C^n$ with strongly convex smooth boundary points and let $P_{\Omega}$ be the Bergman projection over it. Then there exists a Bloch function $b$ over $\Omega$ such that the commutator $[\bar{b},P_{\Omega}]$ is not restricted weak-type $(1,1).$
\end{atheorem}

In fact, our second main result states that for a holomorphic function $b $ the commutator $[\bar{b},P]$ is weak-type $(1,1)$ if and only if $b \in H^\infty(\B).$ We remark this result parallels a theorem for the Hilbert transform and $L^\infty$ functions, see \cite[Theorem 3.2]{Accomazzo}.

\begin{atheorem}\label{HInfinity}
Let $b \in \textnormal{Hol}(\B)$. Then the following statements are equivalent:
\begin{enumerate}
    \item $[\bar{b},P]$ is weak-type $(1,1)$;
    \item $b \in H^\infty(\B)$.
\end{enumerate}
\end{atheorem}

Our third main result is a version of P\'{e}rez's distributional estimate, but with the Bergman projection replacing the Hilbert transform. Moreover, we show that this estimate actually gives a characterization of the Bloch space, similar to the $\BMO$ characterization obtained in \cite{Accomazzo}. In fact, we show that the distributional estimate holding on all bounded, compactly supported functions is equivalent to it holding only on characteristic functions of Bergman balls. 

\begin{atheorem} \label{BlochCharacterize}
Let $b \in \mathrm{Hol}(\B).$ The following are equivalent:
\begin{enumerate}
\item $b \in \mathcal{B}$.
\item there exists $C>0$ so that the following holds for all bounded, compactly supported functions $f$ on $\B$ and $\lambda>0$:

$$ \left| \left \{ z \in \B: |[\overline{b},P]f(z)|>\lambda \right \}  \right | \leq C  \int_{\B} \frac{|f(z)|}{\lambda} \left( 1+ \log^{+}\left( \frac{|f(z)|}{\lambda}\right) \right) \, dV(z).$$

\item there exists $C>0$ so that the following holds for all measurable sets $A \subset \B$ and $\lambda>1$:

$$ \left| \left \{ z \in \B: |[\overline{b},P](\mathbf{1}_{A})(z)|>\lambda \right \}  \right | \leq C  \frac{|A|}{\lambda}.$$

\item there exists $C, r>0$ so that the following holds for all $z \in \B$ and $\lambda>1$:

$$ \left| \left \{ w \in \B: |[\overline{b},P](\mathbf{1}_{E(z,\hat{r})})(w)|>\lambda \right \}  \right | \leq C  \frac{|E(z,\hat{r})|}{\lambda}.$$

\end{enumerate}

\end{atheorem}
\noindent We remark that the third condition is almost the same as a restricted weak-type $(1,1)$, except for the restriction $\lambda>1.$ In the fourth condition, $E(z,\hat{r})$ denotes a polydisk which is comparable to the Bergman metric ball of radius $r$ centered at $z \in \B$; see Section \ref{Prelim} for details.

Moreover, we obtain a refined endpoint distributional inequality when stronger conditions (exponential oscillation, defined precisely in Section \ref{Prelim}) are imposed on the symbol function $b$ (here we do not need a holomorphic or harmonic assumption because of the strong a priori control on the symbol's oscillation).

\begin{atheorem} \label{ExpOscThm} Suppose that $b$ belongs to $\osc_{(\exp L)^{1/\varepsilon}}$ for $\varepsilon \in (0,1].$
There exists $C>0$ so that the following holds for all bounded, compactly supported functions $f$ on $\B$ and $\lambda>0$:

$$ \left| \left \{ z \in \B: |[\bar{b},P]f(z)|>\lambda \right \}  \right | \leq C  \|b\|_{\osc_{(\exp L)^{1/\varepsilon}}} \int_{\B} \frac{|f(z)|}{\lambda} \left( 1+ \log^{+}\left( \frac{|f(z)| \|b\|_{\osc_{(\exp L)^{1/\varepsilon}}}}{\lambda}\right)\right)^\varepsilon \, dV(z).$$
    
\end{atheorem}

Finally, we obtain an estimate of a different flavor showing that the commutator maps the Orlicz space $\LlogL$ to $L^{1,\infty}$, which parallels the endpoint estimate for the commutator of the Bergman projection on the bidisk proven in \cite{HW}. Note that for this estimate, we only need to assume that $b \in \BMO$ is harmonic (not anti-holomorphic). 

\begin{atheorem} \label{WeakTypeMod}
Suppose that $b$ is harmonic on $\B$ and belongs to $\BMO.$ Then, the commutator of the Bergman projection $[\bar{b},P]$ satisfies the $\LlogL$ norm inequality 

$$ \|[\bar{b},P]f\|_{L^{1,\infty}(\B)} \leq C \|b\|_{\BMO} \|f\|_{\LlogL}, \quad f \in \LlogL.$$
\end{atheorem}

This paper is organized as follows. In Section \ref{Prelim}, we collect some necessary background information on dyadic structures and geometry of the unit ball, BMO spaces, and Orlicz spaces. In Section \ref{WeakTypeFails}, we prove Theorem \ref{WeakTypeCounterexample}, Theorem \ref{StronglyConvex}, and Theorem \ref{HInfinity}. In Section \ref{NecessitySection}, we prove the necessity of the Bloch condition for holomorphic symbols. In Section \ref{SuffSection}, we complete the proof of Theorems \ref{BlochCharacterize}, \ref{ExpOscThm}, and \ref{WeakTypeMod}. Finally, in Section \ref{FurtherGen} we discuss some potential further generalizations of this work and future research directions. An Appendix is included at the end which contains the proofs of some geometric lemmas.

\section{Preliminaries} \label{Prelim}
\subsection{Notation} For a measurable set $E \subseteq \B$, we use the notation $|E|=V(E)$ to denote its Lebesgue measure. We use the notation $A \lesssim B$ to mean that there exists a constant $C$, independent of quantities except possibly the dimension $n$, so that $A \leq C B.$ We write $A \sim B$ if $A \lesssim B$ and $B \lesssim A.$ If $\|\cdot \|_A, \| \cdot \|_{B}$ are two seminorms (we will abuse language and refer to seminorms as simply norms in this paper), we write $\|\cdot\|_{A} \sim \|\cdot \|_{B}$ to indicate the norms are equivalent. Given any measurable set $A \subseteq \B$ and locally integrable function b, we use the integral average notation:
$$\langle b \rangle_{A}= \frac{1}{|A|} \int_{A} b(z) \, dV(z).$$

\subsection{Carleson Tents and Dyadic Decomposition of the Unit Ball}

Much of the material concerning the geometry of the unit ball and function spaces can be found in \cite{ZhuBallBook}. Given $z \in \B \setminus \{0\}$, we define the Carleson tent
$$ T_z= \bigg \{w \in \B: \bigg | 1-\frac{\langle z,w \rangle}{|z|} \bigg|<1-|z| \bigg \}.$$ When $z=0$, just set $T_z=\B.$ Carleson tents are important geometric objects in the study of the analysis of Bergman spaces because, for example, they form a Muckenhoupt basis characterizing weighted inequalities for the Bergman projection. However, it is slightly more convenient to approximate Carleson tents by certain dyadic sets, so we can use the machinery of modern dyadic harmonic analysis. We turn to that task next.

We briefly summarize the dyadic decomposition of the unit ball constructed in  \cite{ARS2006} or \cite{RTW}. To begin with, we fix two real parameters $\theta_0, \lambda>0$. For convenience, we assume $\theta_0$ is chosen large enough so that $(n+1)\theta_0>1.$  It is known that there exist dyadic sets on the unit sphere $\partial \B$ which induce dyadic tents inside $\B$ which reflect the geometry of the Bergman kernel appropriately. In particular, let $\rho(z,w)$ be the quasi-metric on $\partial \B$ given by $\rho(z,w)=|1-\langle z, w \rangle|$. We can approximate balls in this quasi-metric by finitely many dyadic systems $\{ \mathcal{Q}_{\ell}\}_{\ell=1}^N$, where $N \in \mathbb{Z}_{+}$ only depends on the ambient dimension $n$. Each dyadic system consists of a collection of Borel sets $Q_j^k$, where $j,k \in \mathbb{N}$ and $k$ corresponds to the dyadic scale of size $\delta^k, \delta=e^{-2\theta_0}.$ For each fixed dyadic system $\mathcal{Q}_{\ell}$, two sets $Q_{j_1}^{k_1}$ and $Q_{j_2}^{k_2}$ are either disjoint or one is contained in the other, for each fixed $k$, the $Q_j^k$ partition $\partial \B$, and each $Q_j^k$ is comparable to a quasi-ball in the $\rho$ metric of radius $\delta^k$.

Let $P_k z$ denote the radial projection of a point $z \in \overline{\B} \setminus \{0\}$ onto the Bergman metric sphere $\{z \in \B: d_\beta(z,0)=k \theta_0\}.$ Set $S_j^k= P_k (Q_j^k)$, or the radial projection of the boundary set $Q_j^k$ to the Bergman sphere with radius $k \theta_0$. We then define, for each boundary dyadic system $\mathcal{Q}_{\ell}$, a system of dyadic ``kubes'' $\mathcal{D}_{\ell}=\{K_j^k\}$ via
$$ K_1^0= \{z \in \B: d_\beta(z,0)<\theta_0\},$$
$$K_j^k:= \{z \in \B: k \theta_0 \leq  d_\beta(z,0) < (k+1)\theta_0, \text{ and } P_k z \in S_j^k \}.$$
Each kube has a center $c_j^k \in K_j^k$ (obtained by projecting certain reference points on $\partial \B$). We put a tree structure on the kubes as follows: we say $K_j^{k+1}$ is a child of $K_\ell^k$ if $P_k(c_j^{k+1}) \in K_\ell^k$. We use the notation $J \preceq K$ or $J \in \mathcal{D}(K)$ to indicate that $J$ is a dyadic descendant of $K$, and define the dyadic Carleson tents:
\begin{equation} \wK:= \bigcup_{J \in \mathcal{D}(K)} J. \end{equation}
For convenience, we denote the union of all these dyadic grids 
$$ \mathcal{D}:= \bigcup_{\ell=1}^N \mathcal{D}_\ell.$$
These kubes and dyadic Carleson tents satisfy the following important properties, which are stated in \cite{RTW, StockdaleWagner2023}:
\begin{proposition} \label{DyadicProperties} For each dyadic system, $\mathcal{D}_{\ell}$ $1 \leq \ell \leq N$, the following properties hold:
\begin{enumerate} 
\item The kubes $\{K_j^k\}$ partition $\B$;
\item The dyadic Carleson tents are dyadically nested; that is, either two tents are disjoint or one is contained in the other;
\item If $K_j^k \in \mathcal{D}_\ell$, then $|\wK_j^k| \sim |K_j^k| \sim e^{-2(n+1)\theta_0 k} $
\item Each kube has at most $e^{2 n \theta_0}$ children. 
\end{enumerate}
Moreover, the list of collections $\mathcal{D}_1,\cdots, \mathcal{D}_N$ collectively approximate all Carleson tents in the following precise sense:
\begin{enumerate}
    \item If $z \in \B$, then there exists $K \in \mathcal{D}$ so $\wK \supset T_z$ and $|\wK| \sim |T_z|$;
    \item Conversely, if $K \in \mathcal{D}$, then there exists $z' \in \B$ so $T_{z'} \supset \wK$ and $|T_{z'}|\sim |\wK|.$
\end{enumerate}
\end{proposition}



Let $P^{+}$ denote the positive Bergman operator on $\B$, which is obtained by replacing the Bergman kernel by its modulus:

$$ P^{+}f(z):= c_n \int_{\B} \frac{f(w)}{|1-\langle z, w \rangle|^{n+1}}\, dV(w), \quad f \in L^1(\B), z \in \B.$$
The following pointwise estimate for the Bergman projection is well-known. Essentially, it says that the Bergman projection is pointwise dominated by a simple dyadic averaging operator. 
\begin{lemma}[\cite{PR2013}, Proposition 3.4 or \cite{RTW}, Lemma 5] \label{lem:dyadicdomination}

The following upper bound  holds uniformly for all $f \in L^1(\B)$ and $z \in \B$:

\begin{equation}
|P^+f(z)| \lesssim \sum_{K \in \mathcal{D}} \langle |f| \rangle_{\widehat{K}} \mathbf{1}_{\widehat{K}}(z).
\label{eq:dyadicdomination}
\end{equation}
\end{lemma}

\subsection{The Bergman Metric, Polydisks, and Kor\'{a}nyi Quasi-Balls}
Many of the following results are scattered across the literature or present in the folklore while lacking explicit proofs. We aim to provide this requisite background with the goal of making this paper reasonably self-contained.

For $z,w \in \B$, we denote by $d_\beta(z,w)$ the hyperbolic (Bergman) distance between $z$ and $w$, and for $r>0$ we set $D(z,r)=\{w \in \B: d_\beta(w,z)<r\}.$ In other words, $D(z,r)$ is the ball of radius $r$ centered at $z$ in the hyperbolic metric. Explicitly, 

\begin{equation}
d_\beta(z,w)= \frac{1}{2} \log \left(\frac{1+|\varphi_z(w)|}{1-|\varphi_z(w)|} \right),
\end{equation}
where $\varphi_z$ denotes the involutive automorphism of $\B$ that interchanges $z$ and $0$.

Given $z \in \B \setminus \{0\}$ and $w \in \B$, set $P_z w= \left \langle w, \frac{z}{|z|} \right  \rangle \frac{z}{|z|}$, the orthogonal projection of $w$ onto the one-dimensional complex subspace spanned by $z$, and let $Q_z w= w- P_zw$ denote the projection onto the orthogonal complement. 

We set up an orthonormal coordinate system which respects these projections. In particular, given $z \in \B \setminus \{0\}$, let $e_1=\frac{z}{|z|}$, and let $e_2,\cdots, e_n$ be any orthonormal vectors in $\C^n$ that span the orthogonal complement of $\text{span }(e_1)$. We can think of $e_1$ as the complex radial direction and $e_2,\cdots, e_n$ as the complex tangential directions.  Note that the choice of $e_2,\cdots, e_n$ is arbitrary up to a unitary transformation in the subspace $\mathbb{C}^{n} \ominus \text{span }(e_1) $, but we will think of such a basis as being fixed for each $z \in \B \setminus \{0\}$ for the remainder of the paper. Represent $w$ in these coordinates via $w= \sum_{j=1}^n \xi_j e_j$. For ease of notation in what follows, let  
$$E(z,r)=\left\{ w \in \C^n: |P_z\, w-z|<r, \quad |\xi_j|<r^{1/2} \text{ for } j=2,\cdots,n \right\}.$$ Note that $E(z,r)$ is a polydisk centered at $z$ with radius $r$ in the radial complex direction $e_1$ and radius $r^{1/2}$ in the complex tangential directions $e_2,\cdots, e_n$. 


If $z \in \B \setminus \{0\}$ and $r>0$, by \cite[Exercise 1.1] {ZhuBallBook}, the Bergman metric ball $D(z,r)$ is equal to the following ellipsoid centered at $c$:

 $$ \bigg\{w \in \B: \frac{\left | \left \langle w, \frac{z}{|z|^2} \right  \rangle z-c \right|^2}{R^2 \sigma^2}+ \frac{\left|w- \left \langle w, \frac{z}{|z|^2}  \right  \rangle z   \right| ^2}{R^2 \sigma}<1 \bigg\},$$
 where
$$R=\tanh{r}, \quad c=\frac{(1-R^2)z}{1-R^2|z|^2}, \quad \sigma= \frac{1-|z|^2}{1-R^2 |z|^2}.$$
The following result is well-known to the experts, but we provide a proof with precise constants in the Appendix for completeness. 
\begin{proposition} \label{ellipsoidspolydisks}
Suppose $z \in \B \setminus \{0\}$, $r>0$, $R$, $\sigma$ are as above, and $\tanh{r}\leq \frac{1}{2}$. Then we have the polydisk containments.

$$ E\left(z, \frac{R^2 \sigma}{3n}\right)\subset  D(z,r) \subset E(z,2 R \sigma)=: E(z, \hat{r}). $$
The upper containment holds for all $r>0.$

\end{proposition}

Let $d_K$ denote the Kor\'{a}nyi quasi-metric: 

$$ d_K(z, w):= \begin{cases} \bigg||z|-|w| \bigg|+ \bigg| 1- \frac{\langle z, w \rangle}{|z|\, |w|}\bigg| & z, w \in \B \setminus \{0\}\\ |z|+|w| & z \text{ or } w=0\end{cases}.$$
Given $z \in \B$ and $r>0$, we define the Kor\'{a}nyi ball $B_K(z,r):=\{w \in \B: d_K(z,w)<r\}$.
Equipped with this quasi-metric and the usual Lebesgue measure, the unit ball with Kor\'{a}nyi balls becomes a space of homogeneous type in the sense of Coifman and Weiss in \cite{CoifmanWeiss}. This distance has played an integral role in the study of weighted estimates for the Bergman projection. Balls $B$ in the Kor\'{a}nyi quasi-metric  that satisfy $ \bar{B} \cap \partial \B \neq \emptyset$ provide an equivalent formulation of Carleson tents; see, for example, \cite{Bekolle, BekolleBesov, BekolleVariable}. Additional details are provided in Lemma \ref{Kor'{a}nyiCarleson}.

Balls in the Kor\'{a}nyi metric are equivalent to certain polydisks (and hence also equivalent to appropriately scaled complex ellipsoids), which the following Lemma makes precise. 

\begin{lemma} \label{Kor'{a}nyiContainment}
There exists $c>0$ depending only on the ambient dimension $n$ so the following containments hold for $r \in (0,1), |z|\geq 1/2$:

$$ E(z, cr) \subset B_K(z,r) \subset E(z,2r).$$

\end{lemma}
\noindent We remark in passing that the hypotheses $|z| \geq \frac{1}{2}$, $r \in (0,1)$ are only necessary to assume for the lower containment.

Finally, we need the following geometric lemma relating Carleson tents and Kor\'{a}nyi balls. This lemma shows that we can either consider Carleson tents or a family of Kor\'{a}nyi balls when computing, for example, a $B_p$ weight characteristic. This result seems to be implicit throughout the literature, although we are unaware of a direct statement. 

\begin{lemma} \label{Kor'{a}nyiCarleson}

Given $z \in \B$, there exists a radius $r \geq (1-|z|)$ and Kor\'{a}nyi ball $B_K(z,r)$ so that the Carleson tent $T_z \subset B_K(z,r)$ and $|T_z| \sim B_K(z,r)$, with implicit constants independent of $z$.

Conversely, given $z \in \B$ and $r  \geq (1-|z|)$, there exists $\widetilde{z} \in \B$ so $T_{\widetilde{z}} \supset B_K(z,r)$ and $|T_{\widetilde{z}}| \sim |B_K(z,r)|$ with constants independent of $z$. 
\end{lemma}

\subsection{The Bloch Space, BMO, and Oscillation Spaces}

The Bloch space $\Bl$ of holomorphic functions on $\B$ is a classical space of holomorphic functions defined by the condition

$$ \Bl= \{ f \in \text{Hol}(\B): \sup_{z \in \B} |\nabla f(z)|(1-|z|)<\infty   \},$$
and we put $\|f\|_{\Bl}= \sup_{z \in \B} | \nabla f(z)|(1-|z|).$ Here, $\nabla$ denotes the holomorphic gradient.

We define $\text{BMO}$ to be the space of locally integrable functions $b$ on $\B$ which satisfy

$$ \|b\|_{\text{BMO}} := \sup_{z \in \B} \frac{1}{|T_z|} \int_{T_z} |b- \langle b \rangle_{T_z}| \, dV< \infty.$$
Since any Carleson tent is well-approximated by a dyadic tent, the BMO norm can be equivalently expressed testing only on dyadic tents in finitely many collections:
$$ \|b\|_{\text{BMO}} \sim \sup_{K \in \mathcal{D}}\frac{1}{|\wK|} \int_{\wK} |b- \langle b \rangle_{\wK}| \, dV $$

This Bergman form of BMO is typically defined by averaging over balls in the hyperbolic metric rather than Carleson tents, but they are equivalent. In particular, if $r>0$, we say $b \in \BMO_r$ if $$ \|b\|_{\text{BMO}_r} := \sup_{z \in \B} \frac{1}{|D(z,r)|} \int_{D(z,r)} |b- \langle b \rangle_{D(z,r)}| \, dV< \infty.$$
The following result is well-known to the experts and has proved in more general settings in the literature (see, for example \cite{HHLPW}), but we present a clean sketch of the proof in $\B$ in the appendix to keep the manuscript readable and self-contained. We note this result is similar to \cite[Theorem 5.25]{ZhuBallBook}, but we give a dyadic proof that illuminates the strong exponential decay that makes this property possible.
\begin{proposition} \label{prop:BMOequivalence}
For each $r>0$, there holds $\BMO=\BMO_r$ and moreover
$$ \| \cdot \|_{\BMO} \sim_r \|\cdot \|_{\BMO_r}.$$ 
\end{proposition}

The following proposition clarifies the relationship between Bloch functions and BMO. 
\begin{proposition} \label{prop:BlochBMO}{\cite[Theorem 5.22]{ZhuBallBook}}
There holds $\BMO \cap \textnormal{Hol}(\B)= \mathcal{B}$ and 
$\|\cdot\|_{\BMO} \sim \|\cdot \|_{\mathcal{B}}.$

\end{proposition}

Say a continuous function $b$ on $\B$ belongs to $\BO_r$ for $r>0$ if

$$ \|b\|_{\BO_r}:= \sup \{ |b(z)-b(w)|:  z,w \in \B; \, d_\beta(z,w) \leq r \}<\infty.$$
 It is a fact that $\|b\|_{\BO_r} \sim \|b\|_{BO_s}$ (assuming one of the norms is finite), so we write $\BO$ for the corresponding function space which is independent of $r$ and write $\|b\|_{\BO}= \|b\|_{\BO_1}$ for convenience. See \cite[pp. 330]{BBCZ} for the proof of this fact.

Next we introduce exponential oscillation function spaces, which are refinements of the BMO condition and were first introduced in P\'{e}rez in \cite{Perez} in the setting of $\R^n$. Their appearance in the Bergman setting is new, to our best knowledge.

 \begin{definition}
 Let $1 \leq r< \infty$. We say that a locally integrable function $b$ on $\B$ belongs to $\osc_{(\exp L)^r}$ if there exists $c>0$ so that

$$ \sup_{K \in \mathcal{D}}\frac{1}{|\widehat{K}|} \int_{\widehat{K}} \exp\left( \left | \frac{b(w)-\langle b \rangle_{\widehat{K}}}{c} \right |^r \right) dV(w)<\infty.$$
The infimum of such constants $c$ such that the resulting supremum is bounded by $1$ is denoted by $\|b\|_{\osc_{(\exp L)^r}}.$

 \end{definition}
 It is obvious that different parameter choices of $r$ lead to a nested scale of spaces: if $r_1 \leq r_2,$ then $\osc_{(\exp L)^{r_2}} \subseteq \osc_{(\exp L)^{r_1}}.$ 
 When $r=1$, and $b$ is additionally a holomorphic function on $\B$, we recover the classical Bloch space:

 \begin{proposition} \label{prop:oscspaces}
  There holds $\osc_{(\exp L)} \cap \, \emph{Hol} (\B)= \mathcal B.$ 
 \end{proposition}
 The proof of Proposition \ref{prop:oscspaces} is also presented in the appendix.

Now, suppose $b \in \mathcal{B}$. Since $b$ satisfies the BMO condition
$$ \sup_{B} \frac{1}{|B|} \int_{B} |b-\langle b \rangle_{B}| \, dV< \infty$$
for all Kor\'{a}nyi balls $B$ by the proof of Proposition \ref{prop:oscspaces}, $b$ also satisfies the following John-Nirenberg estimate (since Carleson tents are geometrically equivalent to Kor\'{a}nyi balls by Lemma \ref{Kor'{a}nyiCarleson}): 

 \begin{lemma} \label{lem:JohnNirenberg}
A function $b \in \mathcal{B}$ satisfies the John-Nirenberg inequality on dyadic Carleson cubes $\widehat{K}$ with $ K \in \mathcal{D}$. More precisely, there exist constants $C_1, C_2>0$ so that for any $b \in \mathcal{B}$, $K \in \mathcal{D}$ and $\lambda>0$,

$$|\{z \in\wK: |b(z)-\langle b \rangle_{\wK}|>\lambda\}| \leq C_1 \exp\left(-\frac{C_2  \lambda}{\|b\|_{\mathcal{B}}}\right) |\wK|.$$

 \end{lemma}
\noindent Such facts are all likely known to the experts (see, for example, \cite{Limani2023} for various related results in the unit disk setting), but we include them in the paper for completeness.
 \begin{remark}
One can show that one can replace the hypothesis $b \in \mathcal{B}$ in Lemma \ref{lem:JohnNirenberg} by the weaker condition that $b$ is harmonic on $\B$ and belongs to $\BMO$. Indeed, if $b$ satisfies these conditions, it is straightforward to check that $b \in \BO$ (for example, using the fact that the Berezin transform preserves harmonic functions and the descriptions of BO in \cite{BBCZ}), and then we can apply the argument in Proposition \ref{prop:oscspaces}.

 \end{remark}

 \subsection{Orlicz Spaces}

 If $\Psi:[0, \infty) $ is a Young's function (convex, increasing, $\Psi(0)=0$, and $\Psi(\infty)=\infty$), 
 let $L_{\Psi}(\B)$ denote the space of Lebesgue measurable functions on $\B$ such that there exists $c>0$ with 

 $$\int_{\B} \Psi\left(\frac{|f|}{c}\right) \, dV < \infty.$$ If $f \in L_{\Psi}(\B)$, we define its Luxembourg norm

 $$\|f\|_{L_{\Psi}(\B)}:= \inf \left \{ c>0:   \int_{\B} \Psi\left(\frac{|f|}{c}\right) \, dV\leq 1 \right \} $$
and we define the localized Luxembourg norm average of a function $f$ on a Carleson tent $\widehat{K}$:
 $$\langle |f| \rangle_{\Psi, \wK}:= \inf \left \{ c>0:  \frac{1}{|\wK|} \int_{\wK} \Psi\left(\frac{|f|}{c}\right) \, dV\leq 1 \right \}.$$

\begin{proposition} \label{prop: Orlicz Inequalities}
Let $0< \varepsilon \leq 1$, $\Psi_\varepsilon(t)= t (\log(e+t))^\varepsilon$, and $\Phi_\varepsilon(t)= e^{t^{1/ \varepsilon}}-1.$
The following estimate holds uniformly over $K \in \mathcal{D}$ and $f \in L_\Psi(\wK), g \in L_{\Phi}(\wK) $:

 \begin{equation} \langle |f g| \rangle_{\wK} \lesssim \langle f \rangle_{\Psi_\varepsilon,\wK}  \langle g \rangle_{\Phi_\varepsilon,\wK}. \label{GeneralizedHolder} \end{equation}

We also have the inequality:

 \begin{equation} x(1+\log^+ x)\leq 2x\log(e+x)\leq 2Cx(1+\log^+ x), \quad x \in [0,\infty)   \label{YoungCompar}\end{equation}
where here $ \log^{+} t:= \max \{0, \log t\}.$ Therefore, the Young's function $\Psi_\varepsilon(t)=t(\log(e+t))^\varepsilon$ could be replaced by $\psi_\varepsilon(t):= t (1+ \log^{+}(t))^{\varepsilon}$, and the resulting Orlicz spaces will be the same with equivalent norms.

Finally, the following sub-multiplicative property of $\Psi_\varepsilon$ holds:

\begin{equation} \label{SubMultiply}
\Psi_\varepsilon(xy) \leq 2 \Psi_\varepsilon(x) \Psi_\varepsilon(y), \quad x,y \in [0,\infty).
\end{equation}

\begin{proof}
The H\"{o}lder-type inequality \eqref{GeneralizedHolder} follows from the generalized H\"{o}lder's inequality for Orlicz spaces and direct estimates on the complementary Young's function for $\Psi_\varepsilon.$ The other inequalities follow from direct computation.

\end{proof}

\end{proposition}

Given any Young's function $\Psi$, define the Young's adapted dyadic maximal function

 $$ M_{\Psi}f(z):= \sup_{K \in \mathcal{D}} \langle |f| \rangle_{\Psi, \wK} \mathbf{1}_{\wK}(z).$$ 

 \begin{lemma} \label{YoungMaxDist}
 The Young's maximal function $M_{\Psi}$ satisfies the distributional inequality 

 $$\left| \left \{z \in \B: M_{\Psi} f(z)>\Lambda   \right \} \right| \leq N  \int_{\B} \Psi\left(\frac{|f(z)|}{\Lambda}\right) \, dV(z), \quad \Lambda>0.$$

 \begin{proof}
 Without loss of generality, we can replace $M_{\Psi}$ by $M_{\Psi}^{\mathcal{D}_\ell}$ for $\ell=1,2,\cdots, N$ (note the constant $N$). Fix $\Lambda>0$. Let $\{\wK_j\}$ denote the pairwise disjoint collection of maximal dyadic intervals so that 

 $$ \|f \|_{\Psi,\wK_j} > \Lambda.$$ Then we have, by definition of the Luxembourg norm:

 \begin{align*}
 \left| \left \{z \in \B: M_{\Psi} f(z)>\Lambda   \right \} \right| & = \sum_{j} |\wK_j| \leq  \sum_{j} \int_{\wK_j} \Psi\left(\frac{|f(z)|}{\Lambda}\right) \, dV(z) \\
 & \leq \int_{\B} \Psi\left(\frac{|f(z)|}{\Lambda}\right) \, dV(z).
 \end{align*}    
 \end{proof}
 \end{lemma}

 \section{The failure of weak-type (1,1) estimates} \label{WeakTypeFails}
 \begin{proof} [Proof of Theorem \ref{WeakTypeCounterexample}]
Inspired by the construction in \cite{Perez}, we give an example here to show that $b\in\text{BMO}$ does not imply that $[b,P]$ is of weak-type $(1,1)$. In fact, the specific example given will show that $[b,P]$ does not even satisfy a restricted weak-type estimate.  Let $b(z)=\overline{{\text{Log}}} (1-z_1)$ where $$\text{Log} (1-z_1)=\sum_{j=1}^\infty\frac{z_1^j}{j}.$$ Then $b(z)$ is anti-holomorphic on $\B$. Since for $z=(z_1,\dots,z_n)\in\B$,
\[|(\text{Log})^\prime(1-z_1)|(1-|z_1|^2)=\frac{1-|z_1|^2}{|1-z_1|}\leq 2,\]
we have $\text{Log}(1-z_1)\in \Bl=\text{BMO}\cap\text{Hol}(\B)$ which implies $b\in\text{BMO}$. Let $B(z,r)$ denote the Euclidean ball centered at $z$ of radius $r$. For $z_0=(s,0,\dots,0)$ with $s\in (0,1)$, we have $B(z_0,1-s)\subseteq \B$. Therefore, $b(z)$ is anti-holomorphic on $B(z_0,1-s)$. We set $f(z):=|B(z_0,1-s)|^{-1}1_{B(z_0,1-s)}(z)$ and compute $[b,P](f)(z)$ for $z\in U$:
\begin{align*}
   &[b,P](f)(z)\\=&b(z)P(f)(z)-P(bf)(z)\\=& \frac{\overline{\text{Log}}(1-z_1)}{|B(z_0,1-s)|}\int_{B(z_0,1-s)}\frac{n!dV(w)}{\pi^n(1-\langle z,w\rangle)^{n+1}}-\int_{B(z_0,1-s)}\frac{n!\overline{\text{Log}}(1-w_1)dV(w)}{|B(z_0,1-s)|\pi^{n}(1-\langle z,w\rangle)^{n+1}}.
\end{align*}
By the mean value property of anti-holomorphic functions,
\begin{align*}
    &\frac{\overline{\text{Log}}(1-z_1)}{|B(z_0,1-s)|}\int_{B(z_0,1-s)}\frac{n!dV(w)}{\pi^n(1-\langle z,w\rangle)^{n+1}}=\frac{\overline{\text{Log}}(1-z_1)}{|B(z_0,1-s)|}\frac{n!|B(z_0,1-s)|}{\pi^n(1-z_1s)^{n+1}}=\frac{n!\overline{\text{Log}}(1-z_1)}{\pi^n(1-z_1s)^{n+1}};\\
    &\int_{B(z_0,1-s)}\frac{n!\overline{\text{Log}}(1-w_1)dV(w)}{|B(z_0,1-s)|\pi^{n}(1-\langle z,w\rangle)^{n+1}}=\frac{|B(z_0,1-s)|\overline{\text{Log}}(1-s)}{|B(z_0,1-s)|\pi^n(1-z_1s)^{n+1}}=\frac{n!\overline{\text{Log}}(1-s)}{\pi^n(1-z_1s)^{n+1}}.
\end{align*}
Therefore, $[b,P](f)(z)=\frac{n!\overline{\text{Log}}(\frac{1-z_1}{1-s})}{\pi^n(1-z_1s)^{n+1}}$. Set $1-s=2^{-k}$ and for $0<m<k$, set $$U_{m}=\{z\in \B: 2^{m-k-1}\leq |1-z_1|\leq 2^{m-k}\}.$$
It is not hard to see that \begin{align*}
&U_m\subseteq \{z\in \mathbb C^n: |1-z_1|\leq 2^{m-k},\sum_{j=2}^n|z_j|^2<2^{m-k+1}\};\\&U_m\supseteq \{z\in \mathbb C^n: 2^{m-k-1}\leq|1-z_1|\leq 2(1-|z|)\leq 2^{m-k},\sum_{j=2}^n|z_j|^2<2^{m-k-1}\}.\end{align*} Therefore, $|U_m|\approx 2^{(n+1)(m-k)}$. 

For $z=(z_1,\dots,z_n)\in U_m$, 
\[|1-z_1s|=s|s^{-1}-z_1|.\]
Note that $s\approx 1$ and 
\[2^{m-k}\approx |1-z_1|-s^{-1}(1-s)\leq |s^{-1}-z_1|\leq |1-z_1|+s^{-1}(1-s)\approx 2^{m-k}.\]
Hence, there exists a constant $C$ such that 
\[\frac{1}{|1-z_1s|}\geq C2^{k-m}.\]
Furthermore, 
\[\left|\overline{\text{Log}}\left(\frac{1-z_1}{1-s}\right)\right|\geq \left|\ln\left|\frac{1-z_1}{1-s}\right|\right|\geq m\ln2.\]
Thus by setting $\lambda=\pi^{-n}n!C^{n+1}2^{{(n+1)}(k-m)}m\ln2$, we have 
\[U_{m}\subseteq \{z\in\B: |[b,P](f)(z)|>\lambda\},\]
and 
\begin{align*}
    &\lambda\cdot|\{z\in\B: |[b,P](f)(z)|>\lambda\}|\\\geq& \lambda\cdot |U_m|\\\gtrsim&2^{(n+1)(k-m)}m\cdot 2^{(n+1)(m-k)}\\=&m.
\end{align*}
Setting $m=\frac{k}{2}$ and letting $k\to \infty$ (i.e. $s\to 1^-$) then yield
\begin{align*}
    \lambda\cdot|\{z\in\B: |[b,P](f)(z)|>\lambda\}|\gtrsim k/2\to \infty,
\end{align*}
proving that $[b,P]$ is not of weak-type $(1,1)$.
\end{proof}

\begin{proof}[Proof of Theorem \ref{StronglyConvex}]
Let $q_0$ be a strongly convex smooth boundary point of $\Omega$. Set $q_0=0$ and let the outward normal direction to the boundary $\mathbf {b}\Omega$ at $p_0$ to be the positive $\text{Re} (z_1)$-axis direction. 
Then $\text{Re} (z_1)<0$ for all $z\in \Omega$ and hence $\text{Log}(-z_1)\in \mathrm{Hol}(\Omega)$. We set $b(z)=\overline{\text{Log}(-z_1)}$. For sufficiently small $s=2^{-k}>0$, we set $z_0=(-s,0,\dots,0)$, and let $B_{z_0}\subseteq \Omega$ be a ball centered at $z_0$. Let $f(z)=|B_{z_0}|^{-1}1_{B_{z_0}}(z)$. Then the mean value property yields that 
\[[b,P](f)(z)=\overline{\text{Log}}\left({-z_1}/{s}\right)K_{\Omega}(z,z_0).\]
Set $U_m=\{z\in\Omega: 2^{m-k-1}\leq |z_1|\leq 2^{m-k}\}$ and let $q_m=(-\frac{3}{2}2^{m-k-1},0,\dots,0)$ be the center of $U_m$. Then for $z\in U_m$,
\[|\text{Log}\left({-z_1}/{s}\right)|\geq m\ln 2.\] 
Recall a result of Engli\v{s} \cite[Theorem 0.4]{Englis}: There exists a neighborhood of $q_0$ in $\mathbb C^n$ and functions $\gamma,\phi,\psi\in C^\infty(U\times U)$ such that on $(U\cap\Omega)\times (U\cap\Omega)$ the Bergman kernel $K_{\Omega}$ has an asymptotic expansion:
\begin{equation}\label{7}K_{\Omega}(x,y)=\frac{\phi(x,y)}{\gamma(x,y)^{n+1}}+\psi(x,y)\log \gamma(x,y).\end{equation}
Here $\gamma,\phi$ and $\psi$ satisfies the following conditions:
\begin{enumerate}
    \item $\bar \partial_x\gamma(x,y)$, $\partial_y\gamma(x,y)$ and all their partial derivatives vanish when $x=y$, and similarly for $\phi$ and $\psi$.
    \item $\gamma(x,x),\phi(x,x),\psi(x,x)$ are real-valued and $-\gamma(x,x)$ is a local defining function of $\Omega$. 
    \item $\phi(x,x)>0$ for $x\in U\cap \Omega$.
\end{enumerate}
By choosing $m=k/2$ (i.e. $k-m=k/2$) and letting $k$ to be sufficiently large, $q_m$, $z_0$ and $U_m$ will be close to the strongly (pseudo)convex point $q_0$. We may assume $q_m$, $z_0$ and $U_m$ are contained in $U$. Then from \eqref{7} there exists a polydisk $D\subseteq U_m$ centered at $q_m$ and a constant $C>0$ such that for $z\in D$,
\[|K_{\Omega}(z,z_0)|\geq C |D|^{-1}.\]
Now set $\lambda=\frac{kC\ln 2}{2|D|}$, then 
\begin{align*}
    \lambda\cdot|\{z\in \omega: |[b,P](f)(z)|>\lambda\}|\geq \lambda\cdot |D|=\frac{kC\ln 2}{2}.
\end{align*} Letting $k\to \infty$ (i.e. $s\to 0^+$)
yields Theorem \ref{StronglyConvex}. 
\end{proof}

Next, we prove Theorem \ref{HInfinity}, which shows the only way a weak-type estimate can hold is if $b$ is a bounded function. 

\begin{proof}[Proof of Theorem \ref{HInfinity}]
If $b \in H^\infty(\B)$, it is trivial to see $[\bar b, P]$ is weak-type $(1,1)$ (see Proposition \ref{LInfinityTrivial}). We therefore focus on the converse direction. Let $b \in \text{Hol}(\B)$ and suppose $[\bar b, P]$ is weak-type $(1,1)$. Suppose by way of contradiction that $b$ is unbounded. Since $b$ is holomorphic, there exists a sequence $\{z_k\}$ in $\B$ with limit point on the boundary and $\lim_{k \rightarrow \infty} |b(z_k)|= \infty.$

Consider the sequence of functions $f_k= |B(z_k, r_k)|^{-1} \mathbf{1}_{B(z_k, r_k)} $, where $r_k= \frac{1-|z_k|}{2}.$ Note $\|f_k\|_{L^1(\B)}=1$ for all $k$. A straightforward computation, similar to the proof of Theorem \ref{WeakTypeCounterexample}, shows that
\begin{equation}
[\bar b, P](f_k)(z)= \frac{n! \left(\bar{b}(z)-\bar{b}(z_k)\right)}{\pi^n(1-\langle z, z_k \rangle)^{n+1}}, \quad z \in \B.
\end{equation}

Next, set
$$ \lambda_k= \frac{n!}{\pi^n 2^{n+2}} \inf_{z \in B(0, \frac{1}{2})}  |b(z)-b(z_k)|.$$ Without loss of generality, we may assume $\lambda_k>0$ for all $k$, and by hypothesis $\lim_{k \rightarrow \infty} \lambda_k=\infty.$
Note that if $z \in B(0, \frac{1}{2})$, we have the bound
\begin{equation*}
|[\bar b, P](f_k)(z)|    \geq \frac{n!}{\pi^n 2^{n+1}} |b(z)-b(z_k)|  > \lambda_k.
\end{equation*}
In other words, 

\begin{equation}
B\left(0,\frac{1}{2}\right) \subseteq \{ z \in \B: |[\bar b, P](f_k)(z)|> \lambda_k\}.
\end{equation}
The hypothesized weak-type estimate then implies 
 \begin{equation}
\left|B\left(0,\frac{1}{2}\right) \right | \leq \left| \{ z \in \B: |[\bar b, P](f_k)(z)|> \lambda_k\} \right| \leq \frac{C}{\lambda_k}.
\end{equation} 
Letting $k \rightarrow 0$, the right side tends to $0$ while the left hand side is bounded below by a positive dimensional constant, giving the desired contradiction.
\end{proof}

\section{Necessary conditions for the boundedness} \label{NecessitySection}

We now show that the Bloch condition is a necessary condition for a P\'{e}rez-type distributional estimate for commutators with anti-holomorphic symbols.

\begin{theorem} \label{thm:suff}
Suppose that $b \in \mathrm{Hol}(\B)$ and that there exists $C>0$ so that the following holds for all bounded, compactly supported functions $f$ on $\B$ and $\lambda>0$:
$$ \left| \left \{ z \in \B: |[\bar{b},P]f(z)|>\lambda \right \}  \right | \leq C   \int_{\B} \frac{|f(z)|}{\lambda} \left( 1+ \log^{+}\left( \frac{|f(z)|}{\lambda}\right) \right) \, dV(z).$$
Then $b \in \mathcal{B}$. 


\begin {proof}
Since $b \in \text{Hol}(\B)$, by Proposition \ref{prop:BlochBMO} it is enough to show $b \in \mathrm{BMO}.$ Fix a point $z_0 \in \B$ and $r>0$; all estimates we obtain will be uniform in $z_0$, and it is enough to prove $b \in \textrm{BO}_r$ since these spaces are independent of $r$. We will test on the characteristic function $f(w):= \mathbf{1}_{E(z_0,\hat{r})}(w)$, where we recall from Proposition \ref{ellipsoidspolydisks} that $E(z_0,\hat{r})$ with $\hat{r}=2R \sigma$ denotes the polydisk comparable to  $D(z_0,r)$ hyperbolic (Bergman) ball of radius $r$ centered at $z_0$. First, let us compute

\begin{align*}
P(\mathbf{1}_{E(z_0,\hat{r})})(w) & = c_n \int_{E(z_0,\hat{r})} \frac{1}{(1-\langle w, \zeta \rangle )^{n+1}}\, dV(\zeta) = \frac{c_n\,|E(z_0,\hat{r})|}{(1-\langle w, z_0 \rangle )^{n+1}}
\end{align*}
by the mean-value property for anti-holomorphic functions on polydisks (recall anti-holomorphic functions are pluriharmonic).
 In an entirely similar way,

\begin{align*}
P( \overline{b} \mathbf{1}_{E(z_0,\hat{r})})(w)
& = c_n \int_{E(z_0,\hat{r})} \frac{\bar{b}(\zeta)}{(1-\langle w, \zeta \rangle )^{n+1}}\, dV(\zeta)= \frac{c_n\,|E(z_0,\hat{r})| \bar{b}(z_0)}{(1-\langle w, z_0 \rangle )^{n+1}}.
\end{align*}
Therefore, if we suppose $w \in E(z_0,\hat{r}),$ we actually have the pointwise equivalence, using the standard Bergman metric estimate $|1-\langle w, z_0 \rangle|^{n+1} \sim  (1-|z_0|^2)^{n+1} \sim |E(z_0,\hat{r})|$ for such $w$:

$$ |[\bar{b},P](\mathbf{1}_{E(z_0,\hat{r})})(w)| \sim |b(w)-b(z_0)| ,$$
where the implicit constants depend on $r$ and the ambient dimension $n$ but not $w, z_0,$ or $b$. 

Next, we estimate the $\mathrm{BMO}$-like quantity $\int_{E(z_0,\hat{r})} |b-b(z_0)|^{1/2} \, dV.$ We have

\begin{align*}
& \int_{E(z_0,\hat{r})} |b(w)-b(z_0)|^{1/2} \, dV(w)  = \frac{1}{2} \int_{0}^{\infty} \lambda^{-1/2} |\{w \in E(z_0,\hat{r}):|b(w)-b(z_0))|>\lambda \}| \, d \lambda \\
& \leq \frac{|E(z_0,\hat{r})|}{2} \int_{0}^{1} \lambda^{-1/2} \, d \lambda + \frac{1}{2}  \int_{1}^{\infty} \lambda^{-1/2} \left |\left \{w \in E(z_0,\hat{r}):|[\overline{b},P](\mathbf{1}_{E(z_0,\hat{r})})(w)|>c \lambda \right \} \right | \, d \lambda.
\end{align*}

The first term in the last display is equal to $|E(z_0,\hat{r})|$, while the second term can be estimated by using the hypothesized distributional inequality: 

\begin{align*}
& \frac{1}{2}  \int_{1}^{\infty} \lambda^{-1/2} \left |\left \{w \in E(z_0,\hat{r}):|[\overline{b},P](\mathbf{1}_{E(z_0,\hat{r})})(w)|>c \lambda \right \} \right | \, d \lambda \\
& \lesssim C  \int_{1}^{\infty} \lambda^{-3/2} \int_{E(z_0,\hat{r})} \log \left(e + \frac{1}{c \lambda} \,  \right) \, dV  \, d \lambda \\
& \lesssim |E(z_0,\hat{r})|.
\end{align*}

Since none of these estimates depend on the point $z_0$, and $E(z_0,\hat{r}) \supset D(z_0,r)$ with $|E(z_0,\hat{r})| \sim |D(z_0,r)|$ with constant independent of $z_0$, altogether we have proven

\begin{equation} \sup_{z_0 \in \B} \frac{1}{|D(z_0,r)|} \int_{D(z_0,r)} |b-b(z_0)|^{1/2} \, dV \lesssim C,. \label{BMO1/2} \end{equation}
where $C$ denotes the constant in the right-hand side of the hypothesized distributional estimate.

By \cite[Lemma 2.24]{ZhuBallBook}, for any $w \in D(z_0,\frac{r}{3})$ there holds, using the subharmonicity of $|b(\cdot)-b(z_0)|^{1/2}$ on  $D(w,\frac{r}{3}) \subset D(z_0,r)$:

\begin{equation}
 |b(w)-b(z_0)| \lesssim \left( \frac{1}{|D(w,\frac{r}{3})|}\int_{D(w,\frac{r}{3})} |b(\zeta)-b(z_0)|^{1/2} \, dV \right)^2 \lesssim C^2.
 \end{equation}







This estimate is in fact enough to conclude $b \in \mathrm{BO}$, because we have shown

$$ \sup_{z_0 \in \B} \sup_{w \in D(z_0, \frac{r}{3})} |b(w)-b(z_0)| \lesssim C^2.$$
Therefore, $b$ has bounded oscillation on Bergman balls (precisely, it belongs to $\textrm{BO}_{\frac{r}{3}}$ ) and thus belongs to $\mathrm{BMO}.$

\end{proof}
    
\end{theorem}

\section{Sufficient conditions for boundedness} \label{SuffSection}

The following dyadic bound and splitting for the operator $[\bar{b},P]$ will be used in the remainder of this section. We have
$$ [\bar{b},P](f)(z)= P\bigg(\bar{b}(z) f(\cdot)- \bar{b}(\cdot)f(\cdot)\bigg)(z), \quad z \in \B.$$

We thus can estimate using the positive Bergman operator $P^+$ and Lemma \ref{lem:dyadicdomination}:

\begin{align*}
|[\bar{b},P](f)(z)| & \lesssim P^{+}\bigg(|b(z)-b(\cdot)||f(\cdot)|\bigg)(z)  \\
& \lesssim \sum_{K \in \mathcal{D}} \bigg \langle |b(z)-b(\cdot)| |f(\cdot)| \bigg \rangle_{\wK} \mathbf{1}_{\wK}(z) \\
& \leq \sum_{K \in \mathcal{D}}|b(z)-\langle b \rangle_{\wK}|   \langle |f| \rangle_{\wK} \mathbf{1}_{\wK}(z) + \sum_{K\in \mathcal{D}} \bigg \langle |b-\langle b \rangle_{\wK}| |f| \bigg \rangle_{\wK} \mathbf{1}_{\wK}(z)
:= \mathcal{T}_b f(z) + \mathcal{T}_b^* f(z).
\end{align*}

The following technical lemmas will be needed to prove one of the main theorems, and in particular control $\mathcal{T}_b$ and $\mathcal{T}_b^*$ separately. This line of argument is heavily inspired by ideas in \cite{LMRR2017} for commutators with real-variable Calder\'{o}n-Zygmund operators. For ease of notation in what follows, we let $\Psi(t):= \Psi_1(t)= t (\log (e+t))$.

 \begin{lemma} \label{EstimateTbStar}
Fix a dyadic grid $\mathcal{D}_{\ell}$, $\ell \in \{1,\cdots, N\}$, let $k \in \mathbb{Z}$, and suppose $f$ is bounded and compactly supported. Consider the following collection of Carleson tents:

$$ \mathcal{F}_k= \{\wK: K \in \mathcal{D}_{\ell}: \,  4^{-k-1}< \|f\|_{L^{\Psi}(Q)}\leq 4^{-k} \} .$$

Then for any measurable set $E \subset \B$, there holds

$$ \int_{E} \left(\sum_{\wK \in \mathcal{F}_k} \mathbf{1}_{\wK} \right) \, dV \leq 2^k |E| + \frac{C}{2^{2^k}} \int_{\B} \Psi(4^k |f|) \, dV.$$

\begin{proof}

Consider the sub-collections $\mathcal{F}_{k,m}$ of Carleson tents, defined inductively for non-negative integers $m$:

$$ \mathcal{F}_{k,0}:= \{\wK: \wK \text{ is maximal in }\mathcal{F}_k\},$$

$$ \mathcal{F}_{k,m}:= \left \{\wK: \wK \text{ is maximal in }\mathcal{F}_k \setminus \left( \bigcup_{\ell=0}^{m-1} \mathcal{F}_{k,\ell} \right) \right\}, m \geq 1.$$
Therefore $\mathcal{F}_{k,0}$ consists of the maximal dyadic tents in $\mathcal{F}_k$, $\mathcal{F}_{k,1}$ consists of the maximal dyadic tents strictly contained in members of $\mathcal{F}_{k,0}$, etc.  

We can make this ``layer decomposition'' disjoint via 

$$ E_{K}:=  \widehat{K} \setminus \bigcup_{\substack{\widehat{J} \in \mathcal{F}_{k,m+1}: \\ \widehat{J} \subset \wK}} \widehat{J},     \quad   \wK \in \mathcal{F}_{k,m},$$
noticing that $E_K\cap E_R =\emptyset$ if $\wK, \widehat{R} \in \mathcal{F}_{k}$ and $K \neq R$.

We will also need the ``almost disjoint sets''

$$ \widetilde{E}_K:=  \wK \setminus \bigcup_{\substack{\widehat{J} \in \mathcal{F}_{k,m+5}: \\ \widehat{J} \subset \wK}} \widehat{J},     \quad   \wK \in \mathcal{F}_{k,m}.$$
Note that the sets $\widetilde{E}_K$ for  $\wK 
\in \mathcal{F}_k$ have finite overlap in the precise sense that for $z \in \B,$

$$ \sum_{\wK \in \mathcal{F}_k} \mathbf{1}_{\widetilde{E}_K} (z) \leq 4.  $$
Additionally define the set, for $m \geq 0$ and $\wK\in \mathcal{F}_{k,m}$ 

$$ S_k(K):= \bigcup_{\substack{\widehat{R} \in \mathcal{F}_{k,m+2^k}\\ \widehat{R} \subset \widehat{K}}} \widehat{R} $$
and notice that $S_k$ is in fact a disjoint union since the level in the decomposition is fixed. It is also obvious from volume considerations that 

$$ |S_k(K)| \leq C e^{-2^{k+1}(n+1)\theta_0}|\widehat{K}| \leq 2^{-2^k} |\widehat{K}|,$$
where in the last inequality we use the fact that the parameter $\delta_0$ was chosen so that $(n+1)\theta_0>1.$

Therefore,
\begin{equation}
\sum_{\wK\in \mathcal{F}_k} |E \cap S_k(K)| \leq \sum_{\wK \in \mathcal{F}_k} | S_k(K)| \leq 2^{-2^k} \sum_{\wK \in \mathcal{F}_k} | \wK| . \label{FarDownPart}
\end{equation}
Now by definition, if $\wK \in \mathcal{F}_k$, there holds, using the sub-multiplicative property of $\Psi$ in \eqref{SubMultiply}:

$$ |\wK| \leq \int_{\wK} \Psi( 4^{k+1}|f|) \, dV \leq 16 \int_{\wK} \Psi( 4^{k}|f|) \, dV.$$
But also, by the construction of the collections $\mathcal{F}_{k,m}$ and sets $\widetilde{E}_K$

\begin{align*}
\int_{\wK} \Psi( 4^{k}|f|) \, dV & = \int_{\widetilde{E}_{K}} \Psi( 4^{k}|f|) \, dV + \int_{\wK \setminus \widetilde{E}_{K}} \Psi( 4^{k}|f|) \, dV \\
& \leq \int_{\widetilde{E}_{K}} \Psi( 4^{k}|f|) \, dV  + \sum_{\substack{\widehat{R} \in \mathcal{F}_{k,m+5}: \\ \widehat{R} \subset \wK}} \int_{\widehat{R}} \Psi(4^k |f|)\, dV \\
& \leq \int_{\widetilde{E}_{K}} \Psi( 4^{k}|f|) \, dV  + \sum_{\substack{\widehat{R} \in \mathcal{F}_{k,m+5}: \\ \widehat{R} \subset \wK}} |\widehat{R}| \\
& \leq \int_{\widetilde{E}_{K}} \Psi( 4^{k}|f|) \, dV + \frac{|\wK|}{32}.
\end{align*}
Putting these estimates together yields

$$|\wK| \leq C  \int_{\widetilde{E}_{K}} \Psi( 4^{k}|f|) \, dV.$$

Substituting this estimate into \eqref{FarDownPart}, summing on $\wK \in \mathcal{F}_k$ and using the finite overlap of $\widetilde{E}_{K}$ then yields

$$ \sum_{\wK \in \mathcal{F}_k} |E \cap S_k(K)|  \leq \frac{C}{2^{2^k}} \int_{\B} \Psi(4^k |f|) \, dV.$$

It remains to estimate
$$\sum_{\wK \in \mathcal{F}_k} |E \cap (\wK \setminus S_k(K))| .$$  Note if $\wK \in \mathcal{F}_{k,m}$, 

$$\wK\setminus S_k(K)= \bigcup_{\ell=0}^{2^k-1} \bigcup_{\substack{\widehat{R} \subset \wK: \\ \widehat{R} \in \mathcal{F}_{k,m+\ell} }} E_{R}.$$

We have 

\begin{align*}
\sum_{\wK \in \mathcal{F}_k} 
 |E \cap (\wK \setminus S_k(K))| & \leq \sum_{m=0}^{\infty} \sum_{\wK \in \mathcal{F}_{k,m}} \sum_{\ell=0}^{2^k-1} \sum_{\substack{\widehat{R} \subset \wK: \\ \widehat{R} \in \mathcal{F}_{k,m+\ell} }} |E \cap E_{R}| \\
 & \leq 2^k \sum_{m=0}^{\infty} \sum_{\wK \in \mathcal{F}_{k,m}} |E \cap E_{K}| \\
 & \leq 2^k |E|,
\end{align*}

where we used the disjointness of the $E_{K}$ in the last step. 
    
\end{proof}

 \end{lemma}

 \begin{lemma} \label{EstimateTb}
Suppose that $b$ is harmonic on $\B$ and belongs to $\BMO$. Given a bounded, compactly supported function $f$, define the dyadic operator

$$ \mathcal{T}_b f(z):= \sum_{K \in \mathcal{D}}|b(z)-\langle b \rangle_{\wK}|   \langle |f| \rangle_{\wK} \mathbf{1}_{\wK}(z).$$

Then the following distributional inequality holds:

\begin{equation} \left| \left \{ z \in \B: |[\mathcal{T}_bf(z)|>\lambda \right \}  \right | \leq C \|b\|_{\rm{BMO}} \int_{\B} \frac{|f(z)|}{\lambda} \left( 1+ \log^{+}\left( \frac{|f(z)|\|b\|_{\BMO}}{\lambda}\right) \right) \, dV(z). \label{TbDist} \end{equation}

\begin{proof}
By scaling, assume $\|b\|_{\BMO}=1.$
By homogeneity and the estimate for the Orlicz maximal function in Lemma \ref{YoungMaxDist}, it is enough for us to prove 

$$ \left| \left \{ z \in \B: |\mathcal{T}_bf(z)|>2 \right \} \cap  \left \{ z \in \B: |M_{\Psi}f(z)| \leq 1/4 \right \}  \right | \leq \int_{\B} \Psi(|f(z)|) \, dV(z).$$
Consider the subcollections of kubes $\mathcal{F}_k$ constructed with respect to the function $f$ in Lemma \ref{EstimateTbStar}. Define also for $\wK \in \mathcal{F}_k$

$$O_{K}:= \{z \in \wK: |b(z)-\langle b \rangle_{\wK}|>(3/2)^k\}.$$
By Proposition \ref{lem:JohnNirenberg} (or the remark thereafter), we have the exponential decay
$$ |O_K| \leq C_1 \exp(-C_2 (3/2)^k) |\wK|  .$$

For $z \in E:= \{ \zeta \in \B: |M_{\Psi}f(\zeta)| \leq 1/4  \}$, we split  the operator $\mathcal{T}_b$ in the following way:

\begin{align*}
\mathcal{T}_b f(z) & = \sum_{K \in \mathcal{D}}|b(z)-\langle b \rangle_{\wK}|   \langle |f| \rangle_{\wK} \mathbf{1}_{\wK}(z) \\
& \leq \sum_{k=1}^\infty \sum_{\wK\in \mathcal{F}_k} (3/2)^k \langle |f| \rangle_{\wK} \mathbf{1}_{\wK}(z) + \sum_{k=1}^\infty \sum_{\wK \in \mathcal{F}_k} (|b(z)-\langle b \rangle_{\wK}| \langle |f| \rangle_{\wK} \mathbf{1}_{O_{K}}(z) \\
& := \mathcal{T}_b^1f(z)+ \mathcal{T}_b^2f(z).
\end{align*}

We then have, by set containment

\begin{align*}
 \left| \left \{ z \in \B: |\mathcal{T}_bf(z)|>2 \right \} \cap E \right| & \leq  \left| \left \{ z \in \B: |\mathcal{T}_b^1f(z)|>1 \right \} \cap E \right| + \left| \left \{ z \in \B: |\mathcal{T}_b^2f(z)|>1 \right \} \cap E \right|\\
 & := |E_1|+|E_2|.  
\end{align*}
We then estimate $|E_1|$ by applying Lemma \ref{EstimateTbStar}:

\begin{align*}
|E_1| & \leq \int_{E_1}|\mathcal{T}_b^1 f| \, dV \\
& \lesssim \sum_{k=1}^\infty (3/2)^k 4^{-k} \int_{E_1} \left(\sum_{\wK \in \mathcal{F}_k}  \mathbf{1}_{\wK}(z) \right) \, dV(z) \\
& \leq \left( \sum_{k=1}^{\infty} (3/4)^{k} \right) |E_1|+  \sum_{k=1}^\infty  \frac{C (3/2)^k}{2^{2^k}} \int_{\B} \Psi(4^k |f|) \, dV,
\end{align*}

which altogether implies

$$ |E_1| \lesssim \int_{\B} \Psi( |f|) \, dV .$$
To control $|E_2|$, recall the sets $\widetilde{E}_K$ constructed for dyadic Carleson tents $\wK$ in the proof of Lemma \ref{EstimateTb}. Recall the estimate 
$$|\wK| \leq C  \int_{\widetilde{E}_{K}} \Psi( 4^{k}|f|) \, dV.$$
Then

\begin{align*}
|E_2| & \leq \int_{E_2}|\mathcal{T}_b^2 f| \, dV \\
& \lesssim \sum_{k=1}^\infty  4^{-k} \sum_{\wK \in \mathcal{F}_k} \int_{O_{K}} |b(z)-\langle b \rangle_{\wK}| \, dV(z) \\
& \leq \sum_{k=1}^\infty  4^{-k} \sum_{\wK \in \mathcal{F}_k} \left(\frac{1}{|\wK|}\int_{\wK} |b(z)-\langle b \rangle_{\wK}|^2 \, dV(z)\right)^{1/2} |\wK|^{1/2} |O_{K}|^{1/2} \\
& \leq C \sum_{k=1}^{\infty} 4^{-k}  \exp(-C_2' (3/2)^k) \sum_{\wK\in \mathcal{F}_k } \int_{\widetilde{E}_{K}} \Psi( 4^{k}|f|) \, dV \\
& \leq C\sum_{k=1}^{\infty} 4^{-k}  \exp(-C_2' (3/2)^k)\int_{\B} \Psi( 4^{k}|f|) \, dV \\
& \leq C \sum_{k=1}^{\infty} k \exp(-C_2' (3/2)^k) \int_{\B} \Psi( |f|) \, dV \\
& \lesssim \int_{\B} \Psi( |f|) \, dV,
\end{align*}
which completes the proof.
    
\end{proof}
 \end{lemma}

We turn to the main point of this paper: proving a distributional inequality in the spirit of the work of C. P\'{e}rez in \cite{Perez}. 

\begin{theorem}  \label{BoundedSuffThm2} Suppose that $b$ is harmonic on $\B$ and belongs to $\BMO$. There exists $C>0$ so that the following holds for all bounded, compactly supported functions $f$ on $\B$ and $\lambda>0$:

$$ \left| \left \{ z \in \B: |[\bar{b},P]f(z)|>\lambda \right \}  \right | \leq C  \|b\|_{\rm{BMO}} \int_{\B} \frac{|f(z)|}{\lambda} \left( 1+ \log^{+}\left( \frac{|f(z)| \|b\|_{\BMO}}{\lambda}\right) \right) \, dV(z).$$

\begin{proof}

We will first use Lemma \ref{EstimateTbStar} to prove

\begin{equation} \left| \left \{ z \in \B: |[\mathcal{T}_b^*f(z)|>\lambda \right \}  \right | \leq C \|b\|_{\rm{BMO}} \int_{\B} \frac{|f(z)|}{\lambda} \left( 1+ \log^{+}\left( \frac{|f(z)| \|b\|_{\BMO}}{\lambda}\right) \right) \, dV(z). \label{TbStarDist} \end{equation}

The corresponding estimate for $\mathcal{T}_bf$ in place of $\mathcal{T}_b^*f$ follows from Lemma \ref{EstimateTb}. 

By homogeneity, it suffices to prove \eqref{TbStarDist} for $\lambda=1$. By the generalized H\"{o}lder inequality \eqref{GeneralizedHolder} and the John-Nirenberg theorem for Bloch functions (Proposition \ref{lem:JohnNirenberg}), we have the following pointwise estimate:

\begin{align*}
\mathcal{T}_b^* f(z) & = \sum_{K \in \mathcal{D}} \bigg \langle |b-\langle b \rangle_{\wK}| |f| \bigg \rangle_{\wK} \mathbf{1}_{\wK}(z) \\
& \lesssim \sum_{K \in \mathcal{D}} \|b-\langle b \rangle_{\wK} \|_{\Phi, \wK} \|f\|_{\Psi, \wK} \mathbf{1}_{\wK}(z) \\
& \lesssim \|b\|_{\BMO} \sum_{K \in \mathcal{D}} \|f\|_{\Psi, \wK} \mathbf{1}_{\wK}(z) := \|b\|_{\BMO} \,   \mathcal{A}_{\Psi}f(z).
\end{align*}

Given this pointwise estimate and Lemma \ref{YoungMaxDist}, we may further reduce to proving

$$\left| \left \{ z \in \B: |\mathcal{A}_{\Psi}f(z)|> 1 \right \} \cap \left \{ z \in \B: |M_{\Psi}f(z)| \leq 1/4 \right \} \right | \leq C  \|b\|_{\rm{BMO}} \int_{\B} \Psi(|f(z)|) \, dV(z).$$

Set $E= \left \{ z \in \B: |\mathcal{A}_{\Psi}f(z)|> 1 \right \} \cap  \left \{ z \in \B: |M_{\Psi}f(z)| \leq 1/4 \right \} .$ Using the definition of the level sets $\mathcal{F}_k$ introduced before, we can estimate, using Chebyshev's inequality:

\begin{align*}
& \left| \left \{ z \in \B: |\mathcal{A}_{\Psi}f(z)|> 1 \right \} \cap \left \{ z \in \B: |M_{\Psi}f(z)| \leq 1/4 \right \} \right |\\
& \leq \int_{E} |\mathcal{A}_{\Psi} f(z)| \, dV(z) 
\leq \sum_{k=1}^{\infty} 4^{-k} \int_{E} \left( \sum_{\wK \in \mathcal{F}_k} \mathbf{1}_{\wK}(z) \right) \, dV.
\end{align*}

Then we can apply Lemma \ref{EstimateTbStar} to estimate the last display:

\begin{align*}
& \sum_{k=1}^{\infty} 4^{-k} \int_{\B} \left( \sum_{\wK \in \mathcal{F}_k} \mathbf{1}_{\wK}(z) \right) \, dV \\
& \leq \sum_{k=1}^{\infty } 2^{-k} |E| +  C \sum_{k=1}^{\infty } \frac{k}{2^{2^k}} \int_{\B} \Psi(|f(z)|) \, dV(z).
\end{align*}

Summing on $k$, we altogether have

$$ |E| \leq \frac{1}{2} |E|+ C \int_{\B} \Psi(|f(z)|) \, dV(z) ,$$
which immediately gives the required estimate.

\end{proof}
\end{theorem}

The preceding theorems give us a new, endpoint operator-theoretic characterization of the Bloch space on the unit ball. If we examine the proofs more closely, we also get a simpler criterion which is equivalent in this case to the endpoint estimate (see Theorem \ref{BlochCharacterize}).








\begin{proof}[Proof of Theorem \ref{BlochCharacterize}]
The implication $(1) \implies (2)$ is given by Theorem \ref{BoundedSuffThm2}, while the implication $(2) \implies (3)$ easily follows by testing on characteristic functions for values of $\lambda>1$. Implication $(3) \implies (4) $ is trivial. To prove the implication $(4) \implies (1)$, notice that in the proof of Theorem \ref{thm:suff} (which formally proves $(2) \implies (1)$), we only need the distributional estimate for $\lambda>1$ and we also only need it for characteristic functions of polydisks $E(z,\hat{r})$.
    
\end{proof}



\subsection{A Slightly Different Estimate}

 In this section, we prove Theorem \ref{WeakTypeMod}, which aligns with the estimate for the Bergman projection on the bidisk which is proven in \cite{HuoWick2020}. Here, due to the proof techniques involved, we merely need to assume that the symbol function $b$ is \emph{harmonic}, not anti-holomorphic. In what follows, we write $\LlogL(\B)$ to mean $L_{\varphi}(\B)$, where $\varphi(t):= t \log^{+}(t).$ We first need the following oscillation lemma, which was first used in \cite{HuoLiWagner}. Below, $\widehat{K}_E$ denotes a fixed, small inflation of a dyadic Carleson tent $\widehat{K}$; explicitly one could write
$$ \widehat{K}_E:= \{z \in \B: d_\beta(z, \widehat{K})<\eta\}$$ for $\eta$ small. The exact value of the parameter $\eta$ is not important, the point is that we consider a small, hyperbolic neighborhood. We have the following measure estimate, the routine proof of which we leave to the interested reader.

\begin{proposition}\label{MeasureHyperbolicNeigh}
Let $K \in \mathcal{D}$ and let $r>0$ be given. For any $z \in K$, there holds
 $$|\widehat{K}_E| \sim |\widehat{K}| \sim_r |D(z,r)|$$ with implicit constants independent of the particular Carleson tent.
\end{proposition}

\begin{lemma} \label{OscLemma}
Let $\ell \in \{1,\cdots, N\}.$ There exists $C>0$ so that for any integrable and harmonic function $b$ on $\B$, the following estimate holds for all $K \in \mathcal{D}_{\ell}$:

\begin{equation}
|b(z)- \langle b \rangle_{\widehat{K}}| \mathbf{1}_{\widehat{K}}(z) \lesssim \sum_{\substack{J \in \mathcal{D}_{\ell} \\ \widehat{J} \subseteq \widehat{K} }} \langle |b-\langle b \rangle_{\widehat{J}}| \rangle_{\widehat{J}_E} \mathbf{1}_{\widehat{J}}(z).    
\end{equation}

\begin{proof}
For every $z \in \widehat{K}$, there exists a unique non-negative maximal integer $k$ and $\widehat{K}_j^k \in \mathcal{D}(\widehat{K})$ so that $z \in \widehat{K}_j^k$ (in fact, $z \in K_j^k$ by maximality). Set $j=j_0$ and let $j_1,\cdots, j_m$ be the indices so that we have the nested containments

$$ \widehat{K}_j^k \subset \widehat{K}_{j_1}^{k-1}\subset \cdots \subset \widehat{K}_{j_m}^{k-m}=: \wK. $$ For shorthand, we write
$$ J^q:= \widehat{K}_{j_q}^{k-q}$$
Since $b-\langle b \rangle_{\wK}$ is harmonic, $|b-\langle b\rangle_{\wK}| $ is subharmonic. Notice that if $z \in K_j^k$, we have, by the sub-mean value inequality for subharmonic functions on Bergman metric balls (see the proof in \cite[Lemma 2.24]{ZhuBallBook}), as well as Proposition \ref{MeasureHyperbolicNeigh}:

\begin{equation} \label{SubmeanMod}
    |b(z)- \langle b \rangle_{\widehat{K}}| \leq \frac{C}{|D(z,r)|}\int_{D(z,r) }   |b- \langle b \rangle_{\widehat{K}}| \, dV \lesssim  \bigg \langle |b- \langle b \rangle_{\widehat{K}}|  \bigg \rangle_{(\widehat{K}_j^k)_E}.
\end{equation}


We then can estimate, using \eqref{SubmeanMod} and repeated applications of the triangle inequality:

\begin{align*}
|b(z)- \langle b \rangle_{\widehat{K}}| & \lesssim \bigg \langle |b- \langle b \rangle_{\widehat{K}}|  \bigg \rangle_{J^0_E}  \leq \bigg \langle |b- \langle b \rangle_{J^0}|  \bigg \rangle_{J^0_E} + |\langle b \rangle_{\widehat{K}}- \langle b \rangle_{J^0}|\\ 
& \leq \bigg \langle |b- \langle b \rangle_{J^0}|  \bigg \rangle_{J^0_E} + \sum_{q=1}^m |\langle b \rangle_{J^{q-1}}- \langle b \rangle_{J^q}| \\
& \lesssim \bigg \langle |b- \langle b \rangle_{J^0}|  \bigg \rangle_{J^0_E} + \sum_{q=1}^m\langle |b- \langle b \rangle_{J^q}| \rangle_{J^q} \\
& \lesssim \sum_{\substack{J \in \mathcal{D}_{\ell} \\ \widehat{J} \subseteq \widehat{K} }} \langle |b-\langle b \rangle_{\widehat{J}}| \rangle_{\widehat{J}_E} \mathbf{1}_{\widehat{J}}(z).
\end{align*}

\end{proof}
    
\end{lemma}





\begin{proof}[Proof of Theorem \ref{WeakTypeMod}]

We estimate $\mathcal{T}_b^* f(z)$ first using Lemma \ref{OscLemma}. Note that given any dyadic kube $J \in \mathcal{D}$, it is clear we can find $\widetilde{J} \in \mathcal{D}$ (recall $\mathcal{D}$ is the union of many dyadic systems) so that the hyperbolic $\eta$-neighborhood $\widehat{J}_E$ is contained in $\widehat{\widetilde{J}}$ and has comparable measure. We use this fact to pull the BMO norm out of the sum:

\begin{align*}
\mathcal{T}_b^* f(z) & \lesssim \sum_{K \in \mathcal{D}} \sum_{J \in \mathcal{D}(K)} \langle |b- \langle b \rangle_{\widehat{J}}| \rangle_{\widehat{J}_E} \frac{\int_{\widehat{J}}|f| dV}{|\wK|} \mathbf{1}_{\wK}(z)\\
& \lesssim \|b\|_{\BMO} \sum_{K \in \mathcal{D}}  \bigg  \langle \sum_{J \in \mathcal{D}(K)} \langle |f| \rangle_{\widehat{J}} \mathbf{1}_{\widehat{J}} \bigg \rangle_{\widehat{K}} \mathbf{1}_{\widehat{K}}(z)
\end{align*}

Note that we have the dyadic maximal function estimate,

\begin{align*}
\int_{\wK} \left( \sum_{J \in \mathcal{D}(K)} \langle |f| \rangle_{\widehat{J}} \mathbf{1}_{\widehat{J}} \right) dV & = \sum_{J \in \mathcal{D}(K)} \langle |f| \rangle_{\widehat{J}}  |\widehat{J}| \\
& \lesssim \sum_{J \in \mathcal{D}(K)} \langle |f| \rangle_{\widehat{J}}  \, |J| \\
& \lesssim \int_{\wK} M_{\mathcal{D}}f dV.
\end{align*}

Putting this together, we have the following pointwise estimate on $\mathcal{T}_b^* f(z)$:

$$ \mathcal{T}_b^* f(z) \lesssim \|b\|_{\BMO} \, \mathcal{A}_{\mathcal{D}}(M_{\mathcal{D}}f)(z). $$

Since the dyadic operator $\mathcal{A}_{\mathcal{D}}$ is weak-type $(1,1)$, we get the following chain of estimates:

$$ \|\mathcal{A}_{\mathcal{D}}(M_{\mathcal{D}}f)\|_{L^{1, \infty}} \lesssim \|b\|_{\BMO} \|M_{\mathcal{D}} f\|_{L^1(\B)} \lesssim \|b\|_{\BMO} \|f\|_{\LlogL}.       $$

Note that we are using the fact that the dyadic maximal function $M_{\mathcal D}$ is bounded from $\LlogL$ to $L^1(\B)$. This can be deduced from the facts that $M_{\mathcal D}$ is weak-type $(1,1)$, bounded on $L^2(\B)$, and $|\B|< \infty$ (see \cite[Theorem 4.2]{StockdaleWagner}).

Next, we estimate $\mathcal{T}_b f(z)$, using Lemma \ref{OscLemma} again:

\begin{align*}
\mathcal{T}_b f(z) &  \lesssim  \sum_{K \in \mathcal{D}} \sum_{J \in \mathcal{D}(K)} \langle |f| \rangle_{\wK} \langle |b- \langle b \rangle_{\widehat{J}}| \rangle_{\widehat{J}_E} \mathbf{1}_{\widehat{J}}(z)\\
& \lesssim \|b\|_{\BMO} \sum_{K \in \mathcal{D}}  \langle |f| \rangle_{\wK} \sum_{J \in \mathcal{D}(K)} \mathbf{1}_{\widehat{J}}(z),
\end{align*}

and then 

\begin{align*}
\|\mathcal{T}_b f(z)\|_{L^{1,\infty}(\B)} \leq \|\mathcal{T}_b f(z)\|_{L^1(\B)} & \lesssim \|b\|_{\BMO}  \sum_{K \in \mathcal{D}}  \langle |f| \rangle_{\wK} \sum_{J \in \mathcal{D}(K)}|\widehat{J}| \\
& \lesssim \|b\|_{\BMO}  \sum_{K \in \mathcal{D}}  \langle |f| \rangle_{\wK} |K| \\
& \lesssim \|b\|_{\BMO} \int_{\B} M_{\mathcal{D}}f \, dV \\
& \lesssim \|b\|_{\BMO} \|f\|_{\LlogL}.
\end{align*}

\end{proof}

At least in the special case of holomorphic functions, the estimate in Theorem \ref{BoundedSuffThm2} (or equivalently Theorem \ref{BlochCharacterize}) is stronger than the estimate in Theorem \ref{WeakTypeMod}.


\begin{corollary}
Let $b \in \text{Hol}(\B)$. Suppose there exists $C_1>0$ so that the following holds for all bounded, compactly supported functions $f$ on $\B$ and $\lambda>0$:
\begin{equation} \left| \left \{ z \in \B: |[\bar{b},P]f(z)|>\lambda \right \}  \right | \leq C_1   \int_{\B} \frac{|f(z)|}{\lambda} \left( 1+ \log^{+}\left( \frac{|f(z)|}{\lambda}\right) \right) \, dV(z). \label{Estimate1}\end{equation} Then there exists $C_2>0$ so that for all $ f \in \LlogL$:
\begin{equation} \|[\bar{b},P]f\|_{L^{1,\infty}(\B)} \leq C_2  \|f\|_{\LlogL}, \quad f \in \LlogL. \label{Estimate2}\end{equation}

\begin{proof}
 Suppose estimate \ref{Estimate1} holds. Then Theorem \ref{BlochCharacterize} implies $b \in \mathcal{B}$, and Theorem \eqref{WeakTypeMod} implies estimate \eqref{Estimate2} holds.   
\end{proof}
    
\end{corollary}

 \subsection{Refined Estimates with Exponential BMO Spaces}
\begin{proof} [Proof of Theorem \ref{ExpOscThm}]
The proof proceeds almost identically to the proof of Theorem \ref{BoundedSuffThm2}. The only major difference is that the Young's function $\Psi$ needs to be replaced by $\Psi_\varepsilon$ throughout, and the generalized H\"{o}lder estimate used in Theorem \ref{thm:suff} needs to be applied to $\Psi_\varepsilon$ and $\Phi_\varepsilon$. We have control over $\|b\|_{\Phi_\varepsilon, \wK}$ in this case by the exponential oscillation condition.

\end{proof}

\section{Further Generalizations} \label{FurtherGen}
As we did with Theorem \ref{StronglyConvex} in Section \ref{WeakTypeFails}, we believe that the rest of these results can be generalized on strongly pseudoconvex domains with smooth boundary. Such domains have key structural properties, including dyadic decompositions induced by the horizontal metric on the boundary and asymptotics on the Bergman kernel. We have not checked all the details for such domains, but we fully expect such a generalization can be carried out. Another possible direction is to work out the endpoint theory for the commutator on the Bergman space of the polydisk, which we already know exhibits different endpoint behavior from the unit ball for the Bergman projection operator (recall \cite{HuoWick2020}).  

It would also be quite interesting to investigate analogs of these results on domains where the Bergman projection is bounded in a proper sub-range of $p$ values, such as the Hartogs triangle, worm domains, or tube domains. 
\section{Appendix}

In this section, we provide proofs for many of the geometric lemmas and propositions in Section \ref{Prelim}.

\begin{proof}[Proof of Proposition \ref{ellipsoidspolydisks}]
Take $w \in D(z,r)$. As before, write $w= \sum_{j=1}^n \xi_j e_j$ with respect to the orthonormal system constructed with $e_1=\frac{z}{|z|}.$  We estimate

\begin{align*}
|P_z w-z| & \leq |P_z w-c|+ z-c| \\
& \leq R \sigma + \left|z- \left( \frac{1-R^2}{1-R^2|z|^2}\right)z \right  | \\
& \leq R \sigma + |z| R^2 \left | \frac{1-|z|^2}{1-R^2|z|^2}\right | \\
& \leq R \sigma + R^2 \sigma \\
& \leq 2 R \sigma.
\end{align*}
On the other hand,
$$ |\xi_j| \leq |Q_z w| \leq R \sqrt{\sigma} \leq  \sqrt{2 R \sigma},$$
completing the proof of the containment $D(z,r) \subset E(z,2 R \sigma).$

On the other hand, suppose $w \in E(z, \frac{R^2 \sigma}{3n})$. Estimate

\begin{align*}
|P_z w-c|^2 & \leq 2|P_zw-z|^2+2|z-c|^2\\
& \leq  \frac{2 R^4 \sigma^2}{9 n^2}+ 2|z-c|^2 \\
& \leq \frac{2 R^4 \sigma^2}{9 n^2}+ 2 R^4 \sigma^2\\
& \leq \frac{5 R^2 \sigma^2}{9},
\end{align*}
by using the fact that $0<R \leq \frac{1}{2}.$

For the complex tangential components, we have

\begin{align*}
|Q_z w|^2 & = \sum_{j=2}^n |\xi_j|^2 \\
& \frac{R^2 \sigma(n-1)}{3n} \leq \frac{R^2\sigma}{3}.
\end{align*}
Putting these estimates together shows $w \in D(z,r)$ and completes the proof. 
\end{proof}

\begin{proof} [Proof of Lemma \ref{Kor'{a}nyiContainment}]
Take $w \in B_K(z,r)$ and compute to see 

\begin{align*}
|P_z w-z| & = |z| \left | \frac{\langle z,w \rangle}{|z|^2}-1 \right | \\
& \leq \left |1- \frac{\langle z,w \rangle}{|z||w|} \right | + |z| \left | \frac{\langle z,w \rangle}{|z|^2}-\frac{\langle z,w \rangle}{|z||w|} \right| \\
& \leq r+ |\langle z,w \rangle| \frac{\left | |w|-|z| \right|}{|w||z|} \\
& \leq 2r.
\end{align*}

On the other hand, for $j=2, \cdots, n$, there holds, by the Pythagorean Theorem

\begin{align*}
|\xi_j|^2 & \leq |Q_z w|^2\\
& = |w|^2- \frac{|\langle w, z \rangle|^2}{|z|^2} \\
& = \left( |w|- \frac{|\langle w, z \rangle|}{|z|}  \right) \left( |w|+ \frac{|\langle w, z \rangle|}{|z|}  \right) \\
& \leq  2 |w| \left(1- \frac{|\langle w,z \rangle|}{|w||z|} \right) \\
& \leq 2 \left|1- \frac{\langle w,z \rangle }{|w||z|} \right| \leq 2r.
\end{align*}
These estimates prove the upper desired containment.

For the lower containment, suppose $w \in E(z, cr)$. We estimate, using the fact $|z| \geq 1/2$ and the Pythagorean Theorem:

\begin{align*}
\bigg| |z|-|w| \bigg| & = \frac{\bigg| |z|^2-|w|^2 \bigg|}{\bigg| |z|+|w| \bigg |} \\
& \leq 2 \bigg | |z|^2-|P_zw|^2-|Q_z w|^2 \bigg | \\
& \leq 2 \bigg | |z|^2-|P_zw|^2 \bigg| + 2|Q_z w|^2 \\
& \leq 4 |z-P_z(w)|^2+2(n-1)cr \\
& 4 c^2 r^2+ 2(n-1)cr \\
& \leq (4 c^2+2nc) r.
\end{align*}

Next, 

\begin{align*}
\bigg | 1- \frac{\langle z, w \rangle}{|z| |w|} \bigg | & = \bigg | \frac{z}{|z|}- \frac{\langle z, w \rangle z}{|z|^2 |w|} \bigg | \\
& \leq 2 \bigg | z- \frac{\langle z, w \rangle z}{|z| |w|} \bigg |\\
& \leq 2 |z-P_z(w)|+ 2 |\langle z, w \rangle| \frac{\bigg||z|-|w|\bigg |}{|z||w|}\\
& \leq (2c+8 c^2+4nc) r
\end{align*}
and we see clearly that with an appropriately small choice of $c$ we can guarantee 

\begin{equation}
\bigg| |z|-|w| \bigg| + \bigg | 1- \frac{\langle z, w \rangle}{|z| |w|} \bigg | \leq r,   
\end{equation}
as desired.
\end{proof}

\begin{proof} [Proof of Lemma \ref{Kor'{a}nyiCarleson}]
We first prove that an arbitrary $T_z$ is contained in some Kor\'{a}nyi ball. Without loss of generality, we can assume $|z|>1/2$. We estimate, for $w \in T_z$:

\begin{align*}
\bigg | |z|-|w| \bigg | & \leq 2 \bigg | |z|^2-|w|^2 \bigg | \\ 
& = 2 \bigg | |z|^2- |P_z w|^2 +|P_zw|^2-|w|^2 \bigg | \\
& \leq 4 |z-P_z w|^2 + 2 |Q_zw|^2,
\end{align*}
so we have reduced matters to estimating $|z-P_zw|^2 $ and $|Q_z w|^2$ separately. We have

\begin{align*}
|z-P_zw|^2 & = |z|^2 \bigg | 1- \frac{\langle w,z \rangle}{|z|^2}\bigg|^2 \\
& \leq 2 \bigg | 1- \frac{\langle w,z \rangle}{|z|}\bigg|^2 + 2 |z|^2 \bigg |\frac{\langle w, z \rangle}{|z|}-\frac{\langle w, z \rangle}{|z|^2} \bigg|^2 \\
& \leq 4 (1-|z|)^2 \leq 2 (1-|z|),
\end{align*}
giving the desired control. Note we used the condition $w \in T_z$ in the penultimate inequality.

Next, using the estimate on $|Q_z w|^2$ in the proof of Lemma \ref{Kor'{a}nyiContainment}, we get

\begin{align*}
|Q_zw|^2 & \leq 2 \left( |w|-\frac{|\langle z,w \rangle|}{|z|}\right)\\
& \leq 2 \left( 1-\frac{|\langle z,w \rangle|}{|z|}\right) \\
& \leq 2\left | 1-\frac{\langle z,w \rangle}{|z|}\right | \leq 2 (1-|z|),
\end{align*}
again using $w \in T_z$. This proves the desired estimate for $\bigg| |z|-|w| \bigg|$.

On the other hand, 
\begin{align*}
\bigg | 1- \frac{\langle z,w \rangle}{|z||w|}\bigg | & \leq | 1- \frac{\langle z,w \rangle}{|z|}\bigg | + (1-|w|) \\
& \leq 2(1-|z|) +\bigg ||z|-|w| \bigg| \\
& \lesssim 1-|z|,
\end{align*}
using the previous estimate we proved. This shows $T_z \subset B_K(z,r)$ where $r=C(1-|z|), \, C \geq 1$, as desired. It is clear $|T_z| \sim |B_K(z,r)| \sim (1-|z|)^{n+1} $ as well. 

Conversely, suppose we are given a Kor\'{a}nyi ball $B_K(z,r)$, where $r \geq (1-|z|)$. Without loss of generality, we can assume $r<\frac{1}{4}$ and $|z| > \frac{3}{4}$, for example. Choose $c \in (0,1)$ so that 

$$ r \sim (1-c|z|) \quad \text{and} \quad 2r+(1-|z|) \leq 1-c|z|. $$
Then, setting $\widetilde{z}:= c z$ and letting $w \in B_K(z,r)$, we get 

\begin{align*}
\left | 1- \frac{\langle \widetilde{z},w \rangle}{|\widetilde{z}|} \right| & = \left | 1- \frac{\langle z,w \rangle}{|z|} \right| \\
& \leq  \left| 1- \frac{\langle z,w \rangle}{|z||w|} \right|+(1-|w|)\\
& \leq r+ \bigg| |z|-|w| \bigg|+ (1-|z|) \\
& \leq 2r+(1-|z|) \leq 1-|\widetilde{z}|.
\end{align*}
This proves $w \in T_{\widetilde{z}},$ as required. 
\end{proof}

\begin{proof} [Proof of Proposition \ref{prop:BMOequivalence}]
Let $r>0$. We first take $b \in \BMO$ and show $b \in \BMO_r.$ Let $z \in \B$ be close to the boundary. By using Proposition \ref{ellipsoidspolydisks}, Lemmas \ref{Kor'{a}nyiContainment} and \ref{Kor'{a}nyiCarleson} together with Proposition \ref{DyadicProperties}, we see that for any $z \in \B$, there exists $\ell \in \{1,\dots, N\}$ and $\wK_{z} \in \mathcal{D}_{\ell}$ so that $\wK_z \supset D(z,r)$ and $|\wK_z| \sim |D(z,r)|$ with constants independent of $z$. Therefore, we have:

\begin{align*}
\frac{1}{|D(z,r)|} \int_{D(z,r)}|b- \langle b \rangle_{D(z,r)}| \, dV & \leq \frac{2}{|D(z,r)|} \int_{D(z,r)}|b- \langle b \rangle_{\wK_z}| \, dV \\
& \lesssim \frac{1}{|\wK_z|}\int_{\wK_z} |b-\langle b \rangle_{\wK_z}| \, dV \lesssim \|b\|_{\BMO}.
\end{align*}

Now we assume that $b \in \BMO_r.$ For the proof of the other inequality, we will use the well-known fact that $\BMO_r=\BMO_s$ for all $s>0$ with equivalent norms (see \cite[Theorem A]{ZhuHankel} and its proof). Fix a dyadic kube $K_{j_0}^{k_0} \in \mathcal{D}$ and let $\wK_{j_0}^{k_0}$ denote its corresponding dyadic tent.  We will use another geometric fact: there exists a radius $R>0$ and points $z_{j}^k$ so that if $\widehat{K}_{j}^k \subseteq \widehat{K}_{j_0}^{k_0}$ is a dyadic descendant of $\widehat{K}_{j_0}^{k_0}$, then  $D(z_j^k, R) \supset K_j^k$ and $|D(z_j^k,R)| \sim |K_j^k| $ (see \cite[Lemma 1]{RTW}). Moreover, we will use the facts that $|K_j^k| \sim e^{-2 \theta_0(n+1)(k-k_0)}|K_{j_0}^{k_0}|$ and that each dyadic cube $\widehat{K}_j^k$ has at most $e^{2n \theta_0}$ children.   Then, breaking up the integral using the dyadic decomposition of the Carleson tent $\wK_{j_0}^{k_0}$ and using the triangle inequality repeatedly, we get:

\begin{align*}   
& \frac{1}{|\wK_{j_0}^{k_0}|} \int_{\wK_{j_0}^{k_0}}|b-\langle b \rangle_{\wK_{j_0}^{k_0}}| \, dV\\
& \leq \frac{2}{|\wK_{j_0}^{k_0}|} \int_{\wK_{j_0}^{k_0}}|b-\langle b \rangle_{D(z_{j_0}^{k_0}, R)}| \, dV \\
& = \frac{2}{|\wK_{j_0}^{k_0}|} \sum_{k=k_0}^{\infty} \sum_{j: \wK_j^k \subseteq \wK_{j_0}^{k_0}} \int_{K_{j}^k} |b- \langle b \rangle_{D(z_{j_0}^{k_0},R)}| \, dV \\
& \lesssim \frac{1}{|\wK_{j}^{k}|} \sum_{k=k_0}^{\infty} e^{-2\theta_0(n+1)(k-k_0)} |K_j^k| \sum_{j: \wK_j^k \subseteq \wK_{j_0}^{k_0}} \frac{1}{|D(z_j^k),R)|}\int_{D(z_j^k,R)} |b- \langle b \rangle_{D(z_{j_0}^{k_0},R)}| \, dV \\
& \lesssim \sum_{k=k_0}^{\infty} e^{-2\theta_0(n+1)(k-k_0)}  e^{2\theta_0n(k-k_0)}   (k+1) \|b\|_{\BMO_{3R}} \\
& \lesssim \|b\|_{\BMO_{3R} } \sum_{k=k_0}^{\infty} k 2^{-e\theta_0(k-k_0)} \lesssim \|b\|_{\BMO_r}.
\end{align*}
\end{proof}

 \begin{proof} [Proof of Proposition \ref{prop:oscspaces}]

By the elementary inequality $x \leq e^x$ for all $x \geq 0 $, it is obvious $\osc_{(\exp L)} \cap \, \text{Hol} (\B) \subseteq \text{BMO} \cap \, \text{Hol} (\B)= \mathcal B.$   
  For the other direction, it suffices to show that if $b \in \text{BMO} \cap \, \text{Hol} (\B)$, then $b$ satisfies the more classical looking BMO condition

 $$ \sup_{B} \frac{1}{|B|} \int_{B} |b-\langle b \rangle_{B}| \, dV< \infty$$
 for all \emph{Kor\'{a}nyi balls} $B$ contained in $\B$. The exponential estimate will then follow from the classical John-Nirenberg inequality for spaces of homogeneous type, see \cite{CRW}. Since we already know the BMO condition holds over Carleson cubes (equivalently, Kor\'{a}nyi balls that have non-empty intersection with the boundary by Lemma \ref{Kor'{a}nyiCarleson}), it suffices to consider a Kor\'{a}nyi ball $B$ with $D \cap \partial \B = \emptyset$, and we can additionally assume that the radius of $B$ is smaller than a fixed multiple of the distance of its center to the boundary (because otherwise by the doubling property we could approximate the ball in measure by a Kor\'{a}nyi ball intersecting the boundary which contains it). In particular, assume we are considering a ball $B_K(z,s),$ where $s< \frac{1-|z|^2}{24n}$. Apply Proposition \ref{ellipsoidspolydisks} together with Lemma \ref{Kor'{a}nyiContainment} with parameters $R=1/2$, $r= \tanh^{-1}(1/2)$; by our choice of constants, such a quasi-ball is also contained in the hyperbolic ball $D(z,r)$, where $r$ is an absolute constant. Now, since any function $b \in \text{BMO} \cap \, \text{Hol} (\B)= \Bl $ also belongs to the space BO, with $\|b\|_{\rm{BO}} \sim \|b\|_{\BMO} \sim \|b\|_{\mathcal{B}}$ (see \cite[Corollary 1, pp.320]{BBCZ}), we have the estimate

 \begin{align*}
 \frac{1}{|B|} \int_{B} |b-\langle b \rangle_{B}| \, dA & \leq \frac{1}{|B|^2} \int_{B} \int_{B}  |b(z)-b(w)| \, dV(w) \, dV(z)\\
 & \lesssim (1+r) \|b\|_{\BO_r},
 \end{align*}
 and since $r$ is a fixed constant we are done.

  \end{proof}

\small\bibliographystyle{amsalpha}
\bibliography{References}

\end{document}